\documentclass[letterpaper]{amsart}
\pdfoutput=1

\usepackage[utf8]{inputenc}
\usepackage[T1]{fontenc}
\usepackage{lmodern}
\usepackage[english]{babel}
\usepackage{amsmath,amssymb,amsthm}
\usepackage{mathrsfs}
\usepackage[marginratio={1:1,1:1},totalwidth=408pt,totalheight=612pt]{geometry}
\usepackage{microtype}
\usepackage{hyperref}
\usepackage{url}
\usepackage{color}
\usepackage{enumerate}
\usepackage[abbrev]{amsrefs}
\usepackage[all]{xy}
\usepackage{tikz-cd,verbatim}

\definecolor{darkblue}{rgb}{0.0,0.0,0.3}
\hypersetup{colorlinks,breaklinks,
            linkcolor=red,urlcolor=red,
            anchorcolor=red,citecolor=red}
						
						\makeatletter
\newsavebox{\@brx}
\newcommand{\llangle}[1][]{\savebox{\@brx}{\(\m@th{#1\langle}\)}
  \mathopen{\copy\@brx\mkern2mu\kern-0.9\wd\@brx\usebox{\@brx}}}
\newcommand{\rrangle}[1][]{\savebox{\@brx}{\(\m@th{#1\rangle}\)}
  \mathclose{\copy\@brx\mkern2mu\kern-0.9\wd\@brx\usebox{\@brx}}}
\makeatother

\theoremstyle{plain}
\newtheorem{theorem}{Theorem}
\newtheorem{lemma}[theorem]{Lemma}
\newtheorem{proposition}[theorem]{Proposition}
\newtheorem{corollary}[theorem]{Corollary}

\theoremstyle{definition}
\newtheorem{definition}[theorem]{Definition}
\newtheorem{example}[theorem]{Example}
\newtheorem{remark}[theorem]{Remark}
\newtheorem{rem}[theorem]{Remark}

\theoremstyle{remark}

\numberwithin{theorem}{section}
\numberwithin{equation}{section}

\newcommand{\C}{\mathbb{C}}

\newcommand{\Z}{\mathbb{Z}}

\newcommand{\mc}[1]{\mathcal{#1}}
\newcommand{\mS}{\mc{S}}
\newcommand{\mM}{\mc{M}}
\newcommand{\ov}{\overline}
\newcommand{\oo}{\emptyset}
\newcommand{\eps}{\varepsilon}

\begin{document}

\title[$C^*$-simplicity of HNN~extensions and groups~acting~on~trees]{$C^*$-simplicity of HNN~extensions and groups~acting~on~trees}

\author[R.\ S.\ Bryder]{Rasmus Sylvester Bryder}
\address{Department of Mathematics\\University of Copenhagen\\Universitetsparken 5\\2100 Copenhagen\\Denmark}
\email{rbryder@gmail.com}

\author[N.\ A.\ Ivanov]{Nikolay A.\ Ivanov}
\address{Faculty of Mathematics and Informatics\\University of Sofia\\blvd.\ James Bourchier 5\\BG-1164 Sofia\\Bulgaria}
\email{nikolay.antonov.ivanov@gmail.com}

\author[T.\ Omland]{Tron Omland}
\address{Department of Mathematics\\University of Oslo\\NO-0316 Oslo\\Norway
\and
Department of Computer Science\\Oslo Metropolitan University\\NO-0130 Oslo\\Norway}
\email{trono@math.uio.no}

\begin{abstract}
We study non-ascending HNN extensions acting on their Bass-Serre trees,
and characterize $C^*$-simplicity and the unique trace property by means of the kernel and quasi-kernels of the HNN extension in question.
We also present a concrete example of an HNN extension that is a new example of a group that is not $C^*$-simple but does have the unique trace property.
Additionally, we include certain more general results, mostly based on previous work of various authors,
concerning $C^*$-simplicity of groups admitting extreme boundary actions, and in particular, groups acting on trees.
\end{abstract}

\date{April 12, 2019}
\thanks{The first author is supported by a PhD stipend from the Danish National Research Foundation (DNRF) through the Centre for Symmetry and Deformation at University of Copenhagen.
The second author is partially supported by grant 80.10-120/2017 of the Research Fund of the University of Sofia.
The third author is funded by the Research Council of Norway through FRINATEK, project no.~240913.}
\maketitle

\section{Introduction}

A discrete group is said to be \emph{$C^*$-simple} when its reduced $C^*$-algebra, i.e., the $C^*$-algebra associated to its left regular representation is simple.
This property for discrete groups found its primus motor in a paper by Powers, who proved in 1975 that the non-abelian free group on two generators is $C^*$-simple \cite{powersfreegroup}.
Since then, many other examples of $C^*$-simple groups have been found (see Pierre de la Harpe's survey \cite{delaharpesurvey}).
What these groups have in common is that they are all proven to be $C^*$-simple by using variations of Powers' technique.
A common denominator for $C^*$-simple groups is that they are highly non-amenable, in the sense that they have \emph{trivial amenable radical}, i.e., they admit no normal non-trivial amenable subgroups.
Another property related to $C^*$-simplicity of a discrete group is \emph{the unique trace property},
meaning that its reduced group $C^*$-algebra admits a unique tracial state.
The unique trace property also implies triviality of the amenable radical.
An early question of de~la~Harpe \cite{delaharpe1985} was whether there was any connection between the aforementioned properties.
Until recently, no characterizations of $C^*$-simplicity nor the unique trace property were known, nor did there exist examples of groups that only satisfied one of these two properties.
Only in 2014 did Kalantar and Kennedy obtain the first known characterization of $C^*$-simplicity \cite{kalantarkennedy}, and later that year,
Breuillard, Kalantar, Kennedy, and Ozawa gave a characterization of the unique trace property in terms of its amenable radical \cite{bkko}.
By means of the result of Kalantar and Kennedy, de~la~Harpe’s question was finally completely settled in 2015,
when Le~Boudec found examples of non-$C^*$-simple groups with the unique trace property, by examining actions of countable groups on trees \cite{leboudec}.

In the 1960’s, Furstenberg considered what came to be known as boundary actions \cite{furstenberg}, in order to investigate the irreducible unitary representations of a group.
A boundary action is a minimal action of a group on a compact Hausdorff space that is strongly proximal, i.e., the weak$^*$-closure of each orbit in the space of probability measures on the space contains a point mass.
Furstenberg proved that any locally compact group $G$ always admits a universal boundary action,
meaning a compact Hausdorff space $\partial_F G$ that maps uniquely $G$-equivariantly onto any other compact Hausdorff space with a boundary action of the group $G$.
This was the technology used in \cite{kalantarkennedy} to obtain a characterization of $C^*$-simplicity,
where the authors showed that a discrete group $\Gamma$ is $C^*$-simple if and only if the action of $\Gamma$ on the Furstenberg boundary $\partial_F\Gamma$ is topologically free.
The theory of boundary actions was also the main force used in \cite{bkko} to show that the unique trace property of a discrete group is equivalent to the group having trivial amenable radical.
Because of the above characterizations, as well as their close affiliation with non-amenability, $C^*$-simplicity and the unique trace property have since witnessed a spike in interest.

\medskip

In this article, we mainly study HNN extensions and obtain necessary and sufficient conditions for $C^*$-simplicity,
in which we take both a purely algebraic point of view, and a geometric point of view,
giving a partial answer to \cite[Problem~28]{delaharpesurvey}.

First, in Sections~\ref{boundaries} and~\ref{treesect} we include a number of results that hold in more general situations,
for boundary actions and groups acting on trees, mostly inspired from the work of Le~Boudec and Matte~Bon \cite{boudecmb}.
Especially, we investigate the case of when a group acts on a tree via an action that is minimal and of general type (cf.~\cite[Section~4.3]{boudecmb}).
These are fairly weak assumptions, and together they give rise to a boundary action of the group on a natural boundary of the tree.
With this set-up, Proposition~\ref{equiv-tree-top-free} provides a connection to conditions implying $C^*$-simplicity earlier studied by de~la~Harpe,
and Theorem~\ref{fixator-amenable} gives necessary and sufficient conditions for $C^*$-simplicity. 

Specializing further, for a graph of groups one can study its fundamental group and its action on the Bass-Serre tree (see Section~\ref{graphs}).
Graphs of groups with only one edge are the most studied examples,
and their fundamental groups are of two types, either an amalgamated free product or an HNN extension.
The former case was investigated in \cite{ivanovomland} and the latter case is studied in detail in the remaining sections of this paper,
where we define two ``quasi-kernels'' of an HNN extension in order to determine $C^*$-simplicity.
Theorem~\ref{hnnpowers} gives combinatorial properties analog to Proposition~\ref{equiv-tree-top-free},
and Proposition~\ref{hnn-finite} shows that in certain cases, $C^*$-simplicity is equivalent to the group being icc.
A characterization of $C^*$-simplicity in terms of quasi-kernels is then given in Theorem~\ref{hnn-utp-cstar}.

Finally, in Section~\ref{example} we produce a concrete HNN extension construced from locally finite groups,
that is not $C^*$-simple, but has the unique trace property.

\medskip

The authors would like to thank the referee for carefully reading the paper.

\section{Preliminaries on boundary actions and \texorpdfstring{$C^*$}{C*}-simplicity}\label{boundaries}

In this section, we first recall the theory of boundary actions and extreme boundary actions, and their relation to $C^*$-simplicity and the unique trace property.
We employ the terminology of \cite[Section~5]{ivanovomland} to formalize the results.

Let $\Gamma$ be a discrete group with identity element $1$.
Consider the Hilbert space $\ell^2(\Gamma)$ with the standard orthonormal basis $\{\delta_g\}_{g\in\Gamma}$,
and define the left regular representation $\lambda$ of $\Gamma$ on $\ell^2(\Gamma)$ by $\lambda(g)\delta_h=\delta_{gh}$.
The reduced group $C^*$-algebra of $\Gamma$, denoted by $C^*_r(\Gamma)$, is the $C^*$-subalgebra of $B(\ell^2(\Gamma))$ generated by $\lambda(\Gamma)$.
A group $\Gamma$ is called \emph{$C^*$-simple} if $C^*_r(\Gamma)$ is simple, that is, if it has no non-trivial proper two-sided closed ideals.

A state on a unital $C^*$-algebra $B$ is a linear functional $\phi\colon B\to\C$ that is positive,
i.e., $\phi(a)\geq 0$ whenever $a\in B$ and $a\geq 0$, and unital, i.e., $\phi(1)=1$.
A state $\phi$ is called tracial if it satisfies the additional property that $\phi(ab)=\phi(ba)$ for all $a,b\in B$.
There is a canonical faithful tracial state $\tau$ on $C^*_r(\Gamma)$,
namely the vector state associated with $\delta_1$,
that is, $\tau(a)=\langle a\delta_1,\delta_1\rangle$ for all $a\in C^*_r(\Gamma)$.
The group $\Gamma$ is said to have the \emph{unique trace property} if $\tau$ is the only tracial state on $C^*_r(\Gamma)$. 

The first large class of groups that were shown to be $C^*$-simple with the unique trace property,
consisted of the \emph{Powers groups} (see \cite{delaharpe1985} for a definition).

Further, recall that a group $\Gamma$ is \emph{amenable} if there exists a state on $\ell^\infty(\Gamma)$
which is invariant under the left translation action by $\Gamma$.
It is explained in \cite{day1957} that every group $\Gamma$ has a unique maximal normal amenable subgroup $R(\Gamma)$,
called the \emph{amenable radical} of $\Gamma$.
It was shown in \cite[Theorem~1.3]{bkko} that $\Gamma$ has the unique trace property if and only if the amenable radical of $\Gamma$ is trivial.
As a consequence, $C^*$-simplicity is stronger than the unique trace property by \cite{paschkesalinas}, and strictly stronger by \cite[Theorem~A]{leboudec}.

The following study of $C^*$-simplicity is based on the dynamical approach to and characterization of $C^*$-simplicity given in \cite{kalantarkennedy}.
Let $\Gamma$ be a group acting continuously (i.e., by homeomorphisms) on a topological space $X$.
The fixed-point set of an element $g\in \Gamma$ is denoted by $X^g=\{x\in X\mid gx=x\}$.
Let $\Gamma_x=\{g\in \Gamma\mid gx=x\}$ be the stabilizer subgroup of $x\in X$ and let $\Gamma_x^\circ$ denote the subgroup of $\Gamma_x$ consisting of all elements that fix a neighborhood of $x$ pointwise.
We then define $\ker(\Gamma\curvearrowright X)=\{g\in \Gamma\mid X^g=X\}$ and
\[
\operatorname{int}(\Gamma\curvearrowright X)=\langle \Gamma_x^\circ \mid  x\in X\rangle=\langle\{g\in \Gamma\mid X^g\text{ has non-empty interior}\}\rangle.
\]
If $\ker(\Gamma\curvearrowright X)=\{1\}$, the action is \emph{faithful},
and when $\operatorname{int}(\Gamma\curvearrowright X)=\{1\}$, the action is \emph{topologically free}.
If $X$ is a Hausdorff space, then $X^g$ is always closed.

\begin{remark}\label{closuresdontmatter}
Let $\Gamma$ act on a Hausdorff space $X$.
If $D$ is a $\Gamma$-invariant subset of $X$, then so is $\ov{D}$, and we have
\[
\ker(\Gamma\curvearrowright D)=\ker(G\curvearrowright\ov{D}),\quad\operatorname{int}(\Gamma\curvearrowright D)=\operatorname{int}(\Gamma\curvearrowright\ov{D}).
\]
Here the first equality follows from continuity of the action.
If $D^g$ has non-empty interior for some $g\in\Gamma$, there is a non-empty open subset $V$ of $\ov{D}$ such that $V\cap D\subseteq D^g$ which ensures that $V\subseteq\ov{V\cap D}\subseteq\ov{D^g}\subseteq\ov{D}^g$.
Conversely, if $\ov{D}^g$ has non-empty interior, then there is a non-empty open subset $V$ of $\ov{D}$ such that $V\subseteq\ov{D}^g$, meaning that $V\cap D\subseteq D^g$. As $V$ is non-empty, so is $V\cap D$.
\end{remark}

The action of $\Gamma$ on the topological space $X$ is said to be \emph{minimal} if the $\Gamma$-orbit $\Gamma x$ is dense in $X$ for all $x\in X$. Now assume that $X$ is compact Hausdorff,
so that $X$ always admits a minimal and closed $\Gamma$-invariant subset by Zorn's lemma.
By the Riesz representation theorem we may identify the space $\mathcal P(X)$ of Radon probability measures on $X$ with the state space of the unital $C^*$-algebra $C(X)$.
We then say that the action of $\Gamma$ on $X$ is
\emph{strongly proximal} if the weak$^*$-closure of every $\Gamma$-orbit in $\mathcal P(X)$ contains a Dirac measure.

The action is a \emph{boundary action} (and $X$ is said to be a \emph{$\Gamma$-boundary}) if the action of $\Gamma$ on $X$ is minimal and strongly proximal.
Any $\Gamma$-boundary has an isolated point if and only if it is a one-point space. 
The major result \cite[Theorem~6.2]{kalantarkennedy} states (among other things) that
if $\Gamma\curvearrowright X$ is a topologically free boundary action, then $\Gamma$ is $C^*$-simple.


A result due to Furstenberg \cite{furstenberg} states that any discrete group $\Gamma$ admits a universal $\Gamma$-boundary $\partial_F\Gamma$, known as the \emph{Furstenberg boundary},
in the sense that for any $\Gamma$-boundary $X$ there exists a unique $\Gamma$-equivariant continuous surjection $\partial_F\Gamma\to X$.
Then \cite[Theorem~6.2]{kalantarkennedy} says that $\Gamma$ is $C^*$-simple if and only if the action of $\Gamma$ on the Furstenberg boundary $\partial_F\Gamma$ is topologically free,
i.e., if and only if $\operatorname{int}(\Gamma\curvearrowright\partial_F\Gamma)$ is trivial.
Moreover, in \cite[Corollary~8]{furman}, it is proved that $R(\Gamma)=\ker(\Gamma\curvearrowright\partial_F\Gamma)$.
Thus, minimality of the $\Gamma$-action on $\partial_F\Gamma$ implies that $\Gamma$ is amenable if and only if $\partial_F\Gamma$ is a one-point space.

Another, intrinsic, description of $C^*$-simplicity is obtained in \cite{kennedy2015char}.
A subgroup $H$ of a group $\Gamma$ is called \emph{recurrent} if for any net $(g_i)$ in $\Gamma$ there is a subnet $(g_j)$ such that $\bigcap_j g_jHg_j^{-1}\neq\{e\}$.
Then \cite[Theorem~5.3]{kennedy2015char} says that $\Gamma$ is $C^*$-simple if it has no amenable recurrent subgroups.

In \cite[Section~7]{ivanovomland} the \emph{amenablish radical} $AH(\Gamma)$ of a group $\Gamma$ is defined.
It is the smallest normal subgroup of $\Gamma$ producing a $C^*$-simple quotient.
In particular, a group is $C^*$-simple if and only if $AH(\Gamma)=\{1\}$ and is called \emph{amenablish} if $AH(\Gamma)=\Gamma$.
For any group $\Gamma$, the normal subgroup $\operatorname{int}(\Gamma\curvearrowright\partial_F \Gamma)$ is always amenablish.
In \cite[Section~6 and~7]{ivanovomland}, it is explained that the class of amenablish groups is radical and is ``dual'' to the class of $C^*$-simple groups, which is residual,
analogously to the duality between the radical class of amenable groups and the residual class of groups with the unique trace property.

The following result is an adaptation of \cite[Theorem~6.2]{kalantarkennedy} and \cite[Theorem~5.2]{Haagerup2016}.

\begin{proposition}\label{interior-kernel}
Let $X$ be a $\Gamma$-boundary.
The following hold:
\begin{itemize}
\item[(i)] $\Gamma$ has the unique trace property if and only if $\ker(\Gamma\curvearrowright X)$ has the unique trace property.
\item[(ii)] $\Gamma$ is $C^*$-simple if and only if $\operatorname{int}(\Gamma\curvearrowright X)$ is $C^*$-simple.
\end{itemize}
\end{proposition}

\begin{proof}
Every $g\in R(\Gamma)$ fixes every point in $X$ by \cite[Corollary~8]{furman}.
Hence $R(\Gamma)\subseteq\ker(\Gamma\curvearrowright X)$, so by e.g., \cite[Lemma~6.7]{ivanovomland}, we have that $R(\Gamma)=R(\ker(\Gamma\curvearrowright X))$.
Consequently, (i) holds. 

For (ii), we set $N=\operatorname{int}(\Gamma\curvearrowright X)$.
For any $\Gamma$-boundary $Y$, there exists a $\Gamma$-equivariant continuous surjective map $\partial_F \Gamma\to Y$.
If $g\in\Gamma$ is such that $(\partial_F \Gamma)^g$ has non-empty interior, then the set $Y^g$ has non-empty interior by \cite[Lemma~3.2]{bkko}.
Thus, $\operatorname{int}(\Gamma\curvearrowright\partial_F \Gamma)\subseteq\operatorname{int}(\Gamma\curvearrowright Y)$.
Since $N$ is a normal subgroup of $\Gamma$, the boundary action $N\curvearrowright\partial_F N$ extends to a boundary action $\Gamma\curvearrowright\partial_F N$ by \cite[Lemma~5.2]{bkko}.
It follows that $\operatorname{int}(\Gamma\curvearrowright\partial_F \Gamma)\subseteq \operatorname{int}(\Gamma\curvearrowright\partial_F N)\cap N=\operatorname{int}(N\curvearrowright\partial_F N)$.
Hence, if $N$ is $C^*$-simple, then $\Gamma$ is $C^*$-simple.

The converse holds by \cite[Theorem~1.4]{bkko} because $\operatorname{int}(\Gamma\curvearrowright X)$ is a normal subgroup of $\Gamma$.
\end{proof}

For $\Gamma$ and $X$ as above and a non-empty subset $U\subseteq X$, define the fixator subgroup $\Gamma_U=\{ g\in\Gamma \mid gx=x \text{ for all }x\in U \}$.
Since group elements act continuously, it should be clear that $\Gamma_U=\Gamma_{\ov U}$.
In particular, if $V$ is dense in $X$, then $\ker(\Gamma\curvearrowright X) = \Gamma_V\subseteq \Gamma_U$ for all non-empty $U\subseteq X$.

Remark that the notation $\Gamma_C$ for rigid stabilizers used in \cite{boudecmb} coincides with our $\Gamma_U$ for $C=X\setminus U$.


Recall that the \emph{normal closure} in a group $\Gamma$ of a subset $S\subseteq\Gamma$ is the smallest normal subgroup, denoted here by $\llangle S\rrangle$, of $\Gamma$ containing $S$.

\begin{lemma}\label{interior stabilizers}
For any minimal action of a group $\Gamma$ on a Hausdorff space $X$ and any $x\in X$, the interior of $\Gamma\curvearrowright X$ equals the normal closure of $\Gamma_x^\circ$.
\end{lemma}

\begin{proof}
The normal closure of $\Gamma_x^\circ$ is always contained in $\operatorname{int}(\Gamma\curvearrowright X)$.
Now let $y\in X$ and let $g\in\Gamma_y^\circ$.
Then there exists a neighbourhood $U$ of $y$ such that $g$ fixes $U$ pointwise.
Since the action is minimal, we can find an $h\in\Gamma$ such that $hx\in U$.
Then $h^{-1}gh$ fixes $h^{-1}U$ pointwise, so that $h^{-1}gh\in\Gamma_x^\circ$, i.e., $g\in h\Gamma_x^\circ h^{-1}$.
Hence $g$ is contained in the normal closure of $\Gamma_x^\circ$, completing the proof.
\end{proof}

\begin{lemma}\label{lem:top-free}
Let $\Gamma\curvearrowright X$ be a minimal action on a Hausdorff space $X$.
The following are equivalent:
\begin{itemize}
\item[(i)] $\Gamma\curvearrowright X$ is topologically free;
\item[(ii)] $\Gamma_x^\circ$ is trivial for some $x\in X$;
\item[(iii)] $\Gamma_U$ is trivial for all non-empty open $U\subseteq X$.
\end{itemize}
\end{lemma}

\begin{proof}
The implication (i)~$\Longrightarrow$~(ii) is obvious.

If $\Gamma_x^\circ$ is trivial for some $x\in X$, then $\Gamma_x^\circ$ is trivial for all $x\in X$ by Lemma~\ref{interior stabilizers}.
Let $U$ be an arbitrary non-empty open set in $X$ and pick any $x\in U$.
Clearly, $\Gamma_U$ is a subgroup of $\Gamma_x^\circ$, which is trivial.
Thus $\Gamma_U$ is trivial, and we have shown (ii)~$\Longrightarrow$~(iii).

Finally, to prove (iii)~$\Longrightarrow$~(i), assume that $\Gamma\curvearrowright X$ is not topologically free, and choose a non-trivial $g\in\Gamma$ such that $X^g$ has non-empty interior $U$.
Then $g\in \Gamma_U$, which is therefore non-trivial.
\end{proof}

We now say that an action $\Gamma\curvearrowright X$ is \emph{amenably free} if $\Gamma_x^\circ$ is amenable for all $x\in X$.

\begin{lemma}\label{lem:am-free}
Let $\Gamma\curvearrowright X$ be a minimal action on a Hausdorff space $X$.
The following are equivalent:
\begin{itemize}
\item[(i)] $\Gamma\curvearrowright X$ is amenably free;
\item[(ii)] $\Gamma_x^\circ$ is amenable for some $x\in X$;
\item[(iii)] $\Gamma_U$ is amenable for all non-empty open $U\subseteq X$.
\end{itemize}
\end{lemma}

\begin{proof}
The implication (i)~$\Longrightarrow$~(ii) is obvious.

Let $U\subseteq X$ be non-empty and open.
By minimality, there exists $h\in\Gamma$ such that $hx\in U$, and thus $\Gamma_U\subseteq\Gamma_{hx}^\circ=h\Gamma^{\circ}_xh^{-1}$.
Hence, it follows that (ii)~$\Longrightarrow$~(iii).

Finally, to show (iii)~$\Longrightarrow$~(i), we assume that $\Gamma_U$ is amenable for all non-empty open $U\subseteq X$.
Choose an arbitrary $x\in X$ and let $\mathcal V$ be the collection of all open neighbourhoods of $x$.
The collection is clearly closed under finite intersections, so it is downward directed, and therefore $\Gamma_x^\circ$ equals the direct limit $\bigcup_{\mathcal V}\Gamma_U$,
which is amenable.
\end{proof}

\begin{remark}\label{rem:amenably}
(i)
If $\Gamma\curvearrowright X$ is an amenable action on a compact Hausdorff space, then it is amenably free.
Indeed, it follows from the definition of amenability of $\Gamma\curvearrowright X$ that for all $x\in X$,
the action $\Gamma_x\curvearrowright\{x\}$ is amenable,
and therefore $\Gamma_x$ is amenable (see e.g.\ \cite[Sections~4.3-4.4]{brown-ozawa}).
Hence, $\Gamma_x^\circ$ is amenable for all $x\in X$.

(ii)
Let $\Gamma\curvearrowright X$ be a minimal action of a \emph{countable} group $\Gamma$ on a Hausdorff space $X$.
Then there exists $x\in X$ such that $\Gamma_x^\circ=\Gamma_x$ (see \cite[Proposition~2.4]{boudecmb}).
Therefore, in this case, $\Gamma\curvearrowright X$ is topologically free if and only if $\Gamma_x$ is trivial for some $x\in X$,
and amenably free if and only if $\Gamma_x$ is amenable for some $x\in X$.
\end{remark}

\begin{proposition}\label{prop:am-free}
Let $\Gamma\curvearrowright X$ be a minimal action on a compact Hausdorff space $X$.
If $\Gamma$ is $C^*$-simple, then the action is either topologically free or not amenably free.
\end{proposition}

\begin{proof}
This follows from \cite[Theorem~7.1]{bkko} and \cite[Theorem~14~(2)]{ozawahandout}.
\end{proof}

\begin{corollary}
If $\Gamma\curvearrowright X$ is an amenably free boundary action,
then $\Gamma$ is $C^*$-simple if and only if $\Gamma\curvearrowright X$ is topologically free.
\end{corollary}

\begin{proposition}\label{prop:bb}
Let $\Gamma\curvearrowright X$ be a minimal action on a compact Hausdorff space such that $N=\ker(\Gamma\curvearrowright X)$ is $C^*$-simple.
Suppose that $\Gamma_U/N$ is non-amenable for all non-dense open $U\subseteq X$.
Then $\Gamma$ is $C^*$-simple.
\end{proposition}

\begin{proof}
The action of $\Gamma$ on $X$ factors to a faithful action of $\Gamma/N$ on $X$, and for any subset $U\subseteq X$,
the fixator subgroup $\Gamma_U$ in $\Gamma$ satisfies $\Gamma_U/N\cong(\Gamma/N)_U$.
Hence $(\Gamma/N)_U$ is non-amenable for all non-dense open $U\subseteq X$,
and it now follows from \cite[Corollary~3.6]{boudecmb} 
and \cite[Theorem~4.1]{kennedy2015char}  
that $\Gamma/N$ is $C^*$-simple
(note also that the countability assumption used in \cite[Lemma~2.1]{boudecmb} is not necessary, as shown in \cite[Proposition~5.2]{kennedy2015char}).
Hence, since $N$ and $\Gamma/N$ are both $C^*$-simple, $\Gamma$ is $C^*$-simple by \cite[Theorem~1.4]{bkko}.
\end{proof}

\begin{corollary}\label{cor:faithful}
Let $\Gamma\curvearrowright X$ be a faithful boundary action that is \emph{not} topologically free.
Consider the following properties:
\begin{itemize}
\item[(i)] $\Gamma_U$ is non-amenable for all non-empty open $U\subseteq X$;
\item[(ii)] $\Gamma$ is $C^*$-simple;
\item[(iii)] $\Gamma_U$ is non-amenable for some non-empty open $U\subseteq X$.
\end{itemize}
Then \textup{(i)}~$\Longrightarrow$~\textup{(ii)}~$\Longrightarrow$~\textup{(iii)}.
\end{corollary}

\begin{proof}
(i)~$\Longrightarrow$~(ii) is just Proposition~\ref{prop:bb} with trivial kernel,
while (ii)~$\Longrightarrow$~(iii) follows from Proposition~\ref{prop:am-free} and Lemma~\ref{lem:am-free}.
\end{proof}

The following definition is based on the notion of an extremely proximal flow from \cite{glasnertop}.
\begin{definition}\label{extboundact}
An action of a group $\Gamma$ on a compact Hausdorff space $X$ (with more than two points) is called an \emph{extreme boundary action}
(called a \emph{strong boundary action} in \cite{lacaspielberg})
if for every closed $K\subsetneq X$ and non-empty open $U\subseteq X$ there exists $g\in \Gamma$ such that $gK\subseteq U$.
\end{definition}

\begin{example}
(a)
Let $\Gamma=G\times H$, where $G\curvearrowright X$ is a faithful extreme boundary action that is topologically free, and $H$ is non-amenable and non-$C^*$-simple.
Then $H=\ker(\Gamma\curvearrowright X)=\operatorname{int}(\Gamma\curvearrowright X)=\Gamma_x^\circ=\Gamma_U$ for all $x\in X$ and non-empty open $U\subseteq X$.
In particular, since $H$ is not $C^*$-simple, $\Gamma$ is not $C^*$-simple.
In other words, when the kernel is not $C^*$-simple, non-amenability of all $\Gamma_x^\circ$ and $\Gamma_U$ is in general not sufficient for $C^*$-simplicity.

(b)
Let $H$ be a non-$C^*$-simple group with the unique trace property admitting a faithful extreme boundary action on some compact Hausdorff space $X$,
for example the group from Section~\ref{example}.
Let $\mathbb{F}_H$ be the free group on the set $H$.
We get an induced action of $\mathbb{F}_H$ on $X$ whose kernel coincides with the kernel of the canoncal surjection $\mathbb{F}_H\to H$ which is $C^*$-simple.
In other words, it is possible that a group $\Gamma$ is $C^*$-simple and $\Gamma/\ker(\Gamma\curvearrowright X)$ is not $C^*$-simple,
i.e., the kernel of the action $\Gamma\curvearrowright X$ is responsible for the $C^*$-simplicity of $\Gamma$.
\end{example}

\begin{lemma}[\cites{glasnertop,lacaspielberg}]\label{eba}
Let $\Gamma\curvearrowright X$ be an extreme boundary action.
Then the action is minimal and strongly proximal, and $X$ is a $\Gamma$-boundary in the sense of Furstenberg.
\end{lemma}

\begin{lemma}\label{open-fixators}
Let $\Gamma\curvearrowright X$ be an extreme boundary action.
Let $U$ and $V$ be two non-empty open sets that are not dense in $X$.
Then the normal closures of $\Gamma_U$ and $\Gamma_V$ coincide, and equal $\operatorname{int}(\Gamma\curvearrowright X)$.
Moreover, $\Gamma_U$ is amenable (resp.\ trivial) if and only if $\Gamma_V$ is amenable (resp.\ trivial).
\end{lemma}

\begin{proof}
By definition, there exists $g\in\Gamma$ such that $g\ov U\subseteq V$.
Let $h\in \Gamma_V$, and let $v\in\ov U$.
Then $g\cdot v\in V$, so $g^{-1}hg\cdot v=g^{-1}g\cdot v=v$.
Since $v$ was arbitrarily chosen, this means that $g^{-1}hg\in \Gamma_{\ov U}$.
Thus $\Gamma_V\subseteq g\Gamma_{\ov U}g^{-1}$, and the normal closure of $\Gamma_V$ is contained in the normal closure of $\Gamma_{\ov U}=\Gamma_U$.
The other inclusion is similar.
Hence, the normal closures of $\Gamma_U$ and $\Gamma_V$ coincide.
If $g\in \Gamma_U$, then $U\subseteq X^g$, so $g\in\operatorname{int}(\Gamma\curvearrowright X)$,
which implies that the normal closure of $\Gamma_U$ is contained in $\operatorname{int}(\Gamma\curvearrowright X)$.
Conversely, if $g\in\Gamma$ and $X^g$ has non-empty interior $W\subseteq X$, then $g\in \Gamma_W$.
Therefore, the conclusion follows from the containments
\[
\{g\in \Gamma\mid X^g\text{ has non-empty interior}\} \subseteq \bigcup \Gamma_W \subseteq \llangle\Gamma_U\rrangle\subseteq \operatorname{int}(\Gamma\curvearrowright X),
\]
where the union is taken over all non-empty open sets of $X$.

Finally, since each of $\Gamma_U$ and $\Gamma_V$ is contained in a conjugate of the other,
we get that $\Gamma_U$ is amenable (resp.\ trivial) if and only if $\Gamma_V$ is amenable (resp.\ trivial).
\end{proof}

\begin{corollary}\label{cor:extension}
Let $\Gamma\curvearrowright X$ be an extreme boundary action such that $N=\ker(\Gamma\curvearrowright X)$ is $C^*$-simple.
If $\Gamma/N\curvearrowright X$ is either topologically free or not amenably free, then $\Gamma$ is $C^*$-simple.

In particular, if $N$ is trivial, then $\Gamma$ is $C^*$-simple if and only if $\Gamma\curvearrowright X$ is either topologically free or not amenably free.
\end{corollary}

\begin{proof}
Note first that $\Gamma/N\curvearrowright X$ is a faithful extreme boundary action.
If the action is topologically free, then $\Gamma/N$ is $C^*$-simple by \cite[Theorem~6.2]{kalantarkennedy}.
If the action is not amenably free, then $\Gamma$ is $C^*$-simple by Proposition~\ref{prop:bb}

The first part follows from Proposition~\ref{prop:bb} and Lemma~\ref{open-fixators},
and the second by Corollary~\ref{cor:faithful} and Lemma~\ref{open-fixators}.
\end{proof}

Note that the second statement above is just \cite[Theorem~1.4]{boudecmb} without the countability condition.
We now get the following characterization of non-$C^*$-simplicity for an extreme boundary action.

\begin{corollary}\label{cor:non-c-simple}
Let $\Gamma\curvearrowright X$ be a faithful extreme boundary action.
The following are equivalent:
\begin{itemize}
\item[(i)] $\Gamma$ is non-$C^*$-simple;
\item[(ii)] $\Gamma_x^\circ$ is non-trivial and amenable for some $x\in X$;
\item[(iii)] $\Gamma_U$ is non-trivial and amenable for some non-empty open $U\subseteq X$.
\end{itemize}
\end{corollary}




\section{Groups acting on trees}\label{treesect}

Let $G$ be a graph, let $V$ be its set of vertices and $E$ its set of edges, and let $s,r\colon E\to V$ denote the source and range maps.
All graphs considered will be unoriented, meaning that we assume the existence of an inversion map $E\to E$, $e\mapsto\ov{e}$, such that $\ov{\ov{e}}=e$ and $s(\ov{e})=r(e)$ for any edge $e\in E$.
The \emph{degree} of a vertex $v\in V$ is the cardinality of $s^{-1}(v)$ or $r^{-1}(v)$ (the numbers coincide), and a vertex is a \emph{leaf} if it has degree~$1$.

A \emph{tree} is a connected graph without circuits.
If $V$ is the set of vertices of a tree $T$, then given any two vertices $u,v\in V$ there are two paths between them (one starting in $u$ and ending in $v$, and one in the opposite direction),
and the (equal) length of these paths is the combinatorial distance $d(u,v)$, yielding the \emph{path metric} on $V$.
For the remainder of this section, we will always consider a tree $T$ with vertices $V$, edges $E$ and source, range and inversion maps as above.
A tree is called \emph{locally finite} if every vertex has finite degree.

A \emph{geodesic} in $T$ between two vertices $v,w\in V$ is a finite sequence $v=v_0,\ldots,v_n=w$ such that $d(v_i,v_j)=\lvert i-j\rvert$ for all $i,j\in\{0,\ldots,n\}$.
A \emph{ray} in $T$ is an infinite sequence of vertices $r=(r_i)_{i\geq 0}$ such that $d(r_i,r_{i+n})=n$ for all $i,n\geq 0$.
Two rays $r,s$ are said to be \emph{cofinal}, written $x\sim y$, if there exist $k,n\geq 0$ such that $r_{i+k}=s_{i+n}$ for all $i\geq 0$.
The set of equivalence classes of cofinal rays in $T$ is denoted $\partial T$, and called the \emph{boundary} of $T$.

In the following, we will identify $T$ with the set of vertices $V$ equipped with its path metric. For an edge $e\in E$, define the \emph{shadow} $Z_0(e)=\{v\in T \mid  d(v,s(e))>d(v,r(e))\}$,
i.e., the set of all vertices that are closer to the range of $e$ than the source.
Moreover, define $Z_\infty(e)$ as the set of all rays $r=(r_i)_{i\geq 0}$ such that $r_j=s(e)$ and $r_{j+1}=r(e)$ for some $j\geq 0$,
and then define $Z_B(e)\subset\partial T$ as $Z_\infty(e)/\sim$.
Finally, define the \emph{extended shadow} $Z(e)=Z_0(e) \cup Z_B(e)$.
The collection $\{Z(e) \mid  e\in E\}$ forms a subbase of compact open sets for a totally disconnected compact Hausdorff topology on $X=T\cup\partial T$,
sometimes called the ``shadow topology'' on $X$.
We refer to \cite[Section~4.1, especially Proposition~4.4]{monodshalom} in this regard (they assume that $T$ is countable, but their proofs hold also without this hypothesis, 
although then the resultant topology is not metrizable, see Appendix~\ref{compactness} for details).

By removing an edge from $T$, we obtain two components, known as \emph{half-trees}.
An \emph{extended half-tree} is a half-tree together with all its associated boundary points.
In this terminology, as explained in \cite[Section~4.3]{boudecmb} (which does not require countability),
half-trees and extended shadows are the same notion, so that the shadow topology is generated by all the extended half-trees in $T \cup \partial T$.

Next, define $F\subseteq T$ as the set of all vertices $v$ of finite degree. 
The following can be deduced from the mentioned sections of \cites{boudecmb,monodshalom}:
\begin{proposition}
All points in $F$ are isolated in $T\cup\partial T$.
The closure $\ov{\partial T}$ of $\partial T$ in $T\cup\partial T$ is $(T\setminus F)\cup\partial T$, and is compact.
Moreover, $\partial T$ is closed in $T\cup\partial T$ if and only if $F=T$, if and only if $T$ is locally finite.
\end{proposition}

\begin{lemma}\label{basissets}
The sets $Z_B(e)$ constitute a basis for the shadow topology restricted to $\partial T$.
\end{lemma}

\begin{proof}
We need to prove that whenever $x\in Z_B(e_1)\cap Z_B(e_2)$ for $x\in\partial T$ and edges $e_1,e_2$ in $T$, then there exists an edge $e$ such that $x\in Z_B(e)\subseteq Z_B(e_1)\cap Z_B(e_2)$.
With this set-up, $x$ is the equivalence class of a ray $(r_i)_{i\geq 0}$ such that $r_0=s(e_1)$ and $r_1=r(e_1)$, which is cofinal to a ray $(s_i)_{i\geq 0}$ such that $s_0=s(e_2)$ and $s_1=r(e_2)$.
Now there exist $k,n\geq 0$ such that $Z_0(e_1)\ni r_{i+k}=s_{i+n}\in Z_0(e_2)$ for all $i\geq 0$. Letting $e$ be the edge defined by $s(e)=r_k$ and $r(e)=r_{k+1}$, then $x$ is clearly contained in $Z_B(e)$.

Any ray in $Z_\infty(e)$ is cofinal to a ray $(t_i)_{i\geq 0}$ satisfying $t_0=r_k=s_n$ and $t_1=r_{k+1}=s_{n+1}$.
For such a ray $(t_i)_{i\geq 0}$, construct a new ray $(t'_i)_{i\geq 0}$ by defining $t'_i=r_i$ for $0\leq i\leq k-1$ and $t'_{k+m}=t_m$ for all $m\geq 0$ in $Z_\infty(e_1)$.
Then $(t'_i)_{i\geq 0}$ is cofinal to $(t_i)_{i\geq 0}$, so that $Z_B(e)\subseteq Z_B(e_1)$. A similar argument shows that $Z_B(e)\subseteq Z_B(e_2)$.
\end{proof}

\begin{lemma}\label{contain-basisset}
Every non-empty open subset of $\ov{\partial T}$ contains $Z(e)\cap\ov{\partial T}$ for some edge $e$ in $T$.
\end{lemma}

\begin{proof}
Notice first that any edge $e$ in $T$ satisfies $\ov{Z_B(e)}=Z(e)\cap\ov{\partial T}$.
Indeed, for any open neighbourhood $V\subseteq\ov{\partial T}$ of a point $x\in Z(e)$,
then $V\cap Z(e)$ is an open neighbourhood of $x\in\ov{\partial T}$ since $Z(e)$ is clopen in the shadow topology
(see Lemma~\ref{here-comes-the-funk}), so that
\[
V\cap Z_B(e)=V\cap Z(e)\cap\partial T\neq\oo.
\]
For edges $e_1,\ldots,e_n$ in $T$ such that $\bigcap_{i=1}^nZ(e_i)$ intersects $\partial T$,
then as $\bigcap_{i=1}^nZ(e_i)\cap\partial T$ is open in $\partial T$, Lemma~\ref{basissets} yields an edge $e$ such that $Z_B(e)\subseteq\bigcap_{i=1}^nZ(e_i)$.
Hence
\[
Z(e)\cap\ov{\partial T}=\ov{Z_B(e)}\subseteq\bigcap_{i=1}^nZ(e_i)\cap\ov{\partial T}
\]
by what we have seen above. Finally, if $U\subseteq\ov{\partial T}$ is a non-empty open set, then $U\cap\partial T\neq\oo$.
As there exist edges $e_1,\ldots,e_n$ in $T$ such that $\bigcap_{i=1}^nZ(e_i)\cap\ov{\partial T}\subseteq U$ and $\bigcap_{i=1}^nZ(e_i)\cap\partial T\neq\oo$,
we infer that $Z(e)\cap\ov{\partial T}\subseteq U$ for some edge $e$ in $T$.
\end{proof}

A \emph{morphism} $\gamma$ of two trees $T_1=(V_1,E_1)$ and $T_2=(V_2,E_2)$ is a tuple of maps $\gamma_V\colon V_1\to V_2$ and $\gamma_E\colon E_1\to E_2$ such that
\[
s(\gamma_E(e))=\gamma_V(s(e)),\ \gamma_E(\ov{e})=\ov{\gamma_E(e)},\quad e\in E_1.\]
If $\gamma_V$ and $\gamma_E$ are bijections, $\gamma$ is called an \emph{isomorphism}, and if $T_1=T_2$ we say that $\gamma$ is an \emph{automorphism}.
The group of automorphisms of the tree $T$ is denoted by $\operatorname{Aut}(T)$.
Furthermore, any automorphism of $T$ is a surjective isometry with respect to the path metric on $T$,
so that we may easily extend $\sigma\in\operatorname{Aut}(T)$ to $T\cup\partial T$
by defining $\sigma(x)$ for any equivalence class $x$ of a ray $(x_i)_{i\geq 0}$ to be the equivalence class of the ray $(\sigma(x_i))_{i\geq 0}$.

We say that an automorphism $\sigma\in\operatorname{Aut}(T)$ is \emph{elliptic} if it fixes a vertex of $T$, and an \emph{inversion} if it fixes no vertices but does fix an edge (i.e., $\sigma(e)=\ov{e}$ for some $e\in E$).
The fixed point set $T^\sigma$ of an elliptic automorphism $\sigma$ of $T$ is easily seen to be a subtree of $T$, and if $\sigma$ is an inversion of $T$, then $\sigma(e)=\ov{e}$ for a unique edge $e\in E$.
An automorphism $\sigma\in\operatorname{Aut}(T)$ is said to be \emph{hyperbolic} if it is neither elliptic nor an inversion.
We will not make the common assumption here that a given group acts without inversions on a tree as most of our results concern fixed points, and an inversion of a tree $T$ fixes no points in $T\cup\partial T$.

For any automorphism $\sigma\in\operatorname{Aut}(T)$, the amplitude of $\sigma$ is $\ell(\sigma)=\min\{d(v,\sigma(v))\mid v\in V\}$.
The characteristic set of $\sigma$ is the $\sigma$-invariant set
\[
T^\sigma=\{v\in V\mid d(v,\sigma(v))=\ell(\sigma)\}.\]
A bi-infinite path in a graph is a subgraph isomorphic to the graph with vertex set $\Z$ and edge set $\{e_n\mid n\in\Z\}\cup\{\ov{e_n}\mid n\in\Z\}$,
with $s(e_n)=r(\ov{e_n})=n$ and $r(e_n)=s(\ov{e_n})=n+1$ for all $n\in\Z$.
A fundamental result of Tits states that for a hyperbolic automorphism $\sigma\in\operatorname{Aut}(T)$,
the characteristic set $T^\sigma$ is always the vertex set of a bi-infinite path $L$ in $T$,
called the \emph{axis} of $\sigma$, and any non-empty subtree of $T$ which is invariant under $\sigma$ and $\sigma^{-1}$ always contains $L$.
We refer to \cite[Proposition~6.4.24]{serre} for details.
Moreover, $\sigma$ admits exactly two fixed points in $T\cup\partial T$, namely the two points in $\partial T$ arising from the $\sigma$-invariant axis of $\sigma$.
Two hyperbolic automorphisms are said to be \emph{transverse} if they have disjoint fixed point sets.

Henceforth, we assume that the action of a discrete group $\Gamma$ on a tree $T$ is \emph{minimal},
i.e., that $T$ contains no proper $\Gamma$-invariant subtrees, and \emph{strongly hyperbolic} (cf.~\cite{dlhpreaux}),
that is, of \emph{general type} in the sense of \cite[Section~4.3]{boudecmb}, meaning that $\Gamma$ contains two transverse hyperbolic automorphisms of $T$.

\begin{rem}\label{no-leaves-needed}
If $\Gamma$ acts minimally on a tree $T$ and $T$ has at least $3$ vertices, then $T$ has no leaves.
Indeed, the subgraph $T'$ arising from removing all vertices of degree~$1$ from $T$ is a $\Gamma$-invariant subtree.
In fact, if we assume additionally that $\Gamma$ contains at least one hyperbolic automorphism of $T$, then $T$ is the union of all axes of hyperbolic automorphisms of $T$ in $\Gamma$
(combine the proof of \cite[Proposition~3.1]{cullermorgan} with \cite[Proposition~8.1]{alperinbass}).
\end{rem}

\begin{lemma}\label{lem:T-boundary}
For a tree $T$, let $\Gamma\curvearrowright T$ be a minimal, strongly hyperbolic action.
Then the induced action $\Gamma\curvearrowright\ov{\partial T}$ is an extreme boundary action.
In particular, $\ov{\partial T}$ is a $\Gamma$-boundary.
\end{lemma}

\begin{proof}
The action of $\Gamma$ on $\ov{\partial T}$ is an extreme boundary action by \cite[Section~4.3]{boudecmb} 
(they also assume faithfulness, but that is not needed in their proof),
and $\ov{\partial T}$ is a $\Gamma$-boundary, by Lemma~\ref{eba}.
\end{proof}

Observe that the fixator subgroup of a half-tree (or shadow) in $T$ coincides with the fixator subgroup of the associated extended half-tree.
Indeed, for any edge $e$ in $T$, $g\in\Gamma$ fixing all vertices in $Z_0(e)$ and any ray in $Z_\infty(e)$, there is $k\geq 0$ such that $gr_i=r_i$ for all $i\geq k$, which implies that $g$ fixes $Z(e)$ pointwise.
In order to determine fixator subgroups of open subsets of the $\Gamma$-boundary $\ov{\partial T}$, we need the following lemma.

\begin{lemma}\label{dontstabmehere}
Let $e$ be an edge in $T$ and suppose that $g\in\Gamma$ fixes $Z_B(e)$ pointwise.
Then $g$ fixes $Z(e)$ pointwise. In particular, the fixator subgroups of $Z_0(e)$, $Z_B(e)$ and $Z(e)\cap\ov{\partial T}$ coincide.
\end{lemma}

\begin{proof}
Notice that $Z(e)\cap\ov{\partial T}=\ov{Z_B(e)}$ as seen in Lemma~\ref{contain-basisset}. Define $v_0=r(e)$ and $v_1=s(e)$.
Now notice that $g$ is not hyperbolic; if it were, then $g$ would have only two fixed points in all of $\ov{\partial T}$,
implying that $Z(e)\cap\ov{\partial T}$ would be an open set in $\ov{\partial T}$ that contained at most two points.
Hence $\ov{\partial T}$ would have an isolated point, contradicting that $\ov{\partial T}$ is a non-trivial $\Gamma$-boundary.
We conclude that $g$ is elliptic (since $g$ cannot be an inversion), so it fixes at least one vertex, say, $v\in T$. Let $m=d(v,v_0)$. 

Notice first that if $(r_i)_{i\geq 0}$ is a ray in $T$ such that $gr_0=r_0$ and there exists $n\geq 0$ such that $r_k\in Z_0(e)$ for all $k\geq n$, then $g$ fixes each vertex $r_i$.
Indeed, there exist $k,n\geq 0$ such that $r_{i+k}=gr_{i+n}$ for all $i\geq 0$, then $k=d(r_k,r_0)=d(gr_n,gr_0)=d(r_n,r_0)=n$.
By geodesics in $T$ being unique (so that the paths $r_0,\ldots,r_k$ and $gr_0,\ldots,gr_k=r_k$ coincide), we conclude that $gr_i=r_i$ for all $i\geq 0$. 
We now have two cases:
\begin{itemize}
\item[(1)] If $v\notin Z_0(e)$, then as $T$ is leafless by Remark~\ref{no-leaves-needed},
any vertex in $Z_0(e)$ belongs to the image of a ray $(r_i)_{i\geq 0}$ such that $r_0=v$, $r_m=v_0$ and $r_k\in Z_0(e)$ for all $k\geq m$. In particular, $g$ fixes $v_0$.
\item[(2)] If $v\in Z_0(e)$, assume first that $v_0$ has degree at least $3$.
Then there exists a ray $(r_i)_{i\geq 0}$ such that $r_0=v$, $r_m=v_0$ and $r_k\in Z_0(e)$ for all $k\geq m$. From the above argument, $g$ fixes $v_0$. 

If $v_0$ has degree~$2$, then let $x_0=v,x_1,\ldots,x_m=v_0$ be the geodesic from $v$ to $v_0$.
Let $0\leq k\leq m$ be smallest such that $x_j$ has degree~$2$ for all $k\leq j\leq m$. If $k>0$, then let $v'\in T$ be a vertex adjacent to $x_{k-1}$, but $v'\neq x_k$.
As $T$ is leafless, there is a ray $(r_i)_{i\geq 0}$ with image in $Z_0(e)$ such that $r_0=x_0=v$, $r_{k-1}=x_{k-1}$ and $r_{k}=v'$.
The same argument as above implies $gv'=v'$ and $gx_{k-1}=x_{k-1}$. Since $g$ is an automorphism fixing $x_{k-1}$ and all vertices adjacent to $x_{k-1}$ bar $x_k$, it must fix $x_k$ as well.
By $x_k$ having degree~$2$, $x_{k+1}$ is also fixed by $g$ (since $g$ fixes $x_{k-1}$ and $x_k$), and we continue this way until we reach $x_m=v_0$, concluding that $g$ fixes $v_0$.
If $k=0$, then let $v'$ be the vertex adjacent to $v$ that is not $x_1$. By taking a ray in $Z_0(e)$ emanating in $v$ and passing through $v'$, we see that $g$ fixes $v'$, so it fixes $x_1$.
By the same method as above, we conclude that $g$ fixes $v_0$.
\end{itemize}

Now, for any $w\in Z_0(e)$, then there exists a ray $(r_i)_{i\geq 0}$ such that $r_0=v_0$, $r_i\in Z_0(e)$ for all $i\geq 0$, and $r_k=w$ for some $k\geq 0$.
As above, we see that $g$ fixes $w$, so that $g$ fixes $Z_0(e)$ and hence $Z(e)$ pointwise.
\end{proof}

Note that by Remark~\ref{closuresdontmatter},
\[
\ker(\Gamma\curvearrowright T)=\ker(\Gamma\curvearrowright\partial T)=\ker(\Gamma\curvearrowright\ov{\partial T}),
\]
using Lemma~\ref{dontstabmehere} for the first identity, and
\[
\operatorname{int}(\Gamma\curvearrowright\partial T)=\operatorname{int}(\Gamma\curvearrowright\ov{\partial T}).
\]

The following is now an immediate consequence of Lemmas~\ref{open-fixators}, \ref{lem:T-boundary}, and~\ref{dontstabmehere}.
\begin{corollary}\label{half-tree-fixators}
Let $T_1$ and $T_2$ be any two half-trees of $T$, and denote by $K_1$ and $K_2$ the fixator subgroups of $T_1$ and $T_2$, respectively.
Then the normal closures of $K_1$ and $K_2$ coincide, and equal $\operatorname{int}(\Gamma\curvearrowright\partial T)$.
Moreover, $K_1$ is amenable if and only if $K_2$ is amenable.
\end{corollary}

We say that a continuous action of a group $G$ on a topological space $X$ is \emph{strongly faithful} \cite[Lemma~4]{delaharpe1985}
if for all finite subsets $F\subseteq G\setminus\{1\}$ there exists $x\in X$ such that $fx\neq x$ for all $f\in F$.
\begin{proposition}\label{equiv-tree-top-free}
The following are equivalent:
\begin{itemize}
\item[(i)] $\operatorname{int}(\Gamma\curvearrowright\partial T)=\{1\}$;
\item[(ii)] $\Gamma\curvearrowright T$ is strongly faithful;
\item[(iii)] $\Gamma\curvearrowright \partial T$ is strongly faithful;
\item[(iv)] $\Gamma\curvearrowright \partial T$ is topologically free, i.e., $\Gamma\curvearrowright T$ is slender in the sense of \cite{dlhpreaux};
\item[(v)] the fixator subgroups of all half-trees are trivial;
\item[(vi)] the fixator subgroup of some half-tree is trivial.
\end{itemize}
If these conditions hold, then $\Gamma$ is a Powers group, and in particular $C^*$-simple.
\end{proposition}

\begin{proof}
The equivalence between (i) and (iv) follows from the definition,
and the equivalence between (i), (v) and (vi) follows from Corollary~\ref{half-tree-fixators}.

(ii) $\Longrightarrow$ (v):
Suppose that $T$ admits a half-tree $T_1$ for which the fixator subgroup $K_1$ is non-trivial.
Let $T_2$ be the component obtained by removing the edge defining $T_1$ from $T$ (and so that $T_1$ and $T_2$ are disjoint), and let $K_2$ be the fixator subgroup of $T_2$. 
Note that $K_1\cap K_2=\ker(\Gamma\curvearrowright T)$.
If $K_1=\ker(\Gamma\curvearrowright T)$, pick $a\in K_1\setminus\{1\}=\ker(\Gamma\curvearrowright T)\setminus\{1\}$, and set $F=\{a\}$. 
Then $a\cdot x=x$ for all $x\in T\cup\partial T$. 
Otherwise, let $a_1\in K_1\setminus\ker(\Gamma\curvearrowright T)$.
Notice that $K_1$ is contained in a conjugate of $K_2$, as seen in the proof of Lemma~\ref{open-fixators}.
Hence $\ker(\Gamma\curvearrowright T)\subsetneq K_1$ implies $\ker(\Gamma\curvearrowright T)\subsetneq K_2$ since $\ker(\Gamma\curvearrowright T)$ is normal.
Let $a_2\in K_2\setminus\ker(\Gamma\curvearrowright T)$, and set $F=\{a_1,a_2\}$.
For every $x\in T\cup\partial T$, then we will either have $a_1\cdot x=x$ or $a_2\cdot x=x$.
In particular, (ii) does not hold.

(iv) $\Longrightarrow$ (iii):
For any finite subset $F\subseteq\Gamma\setminus\{1\}$, $\bigcap_{f\in F}\partial T\setminus(\partial T)^f$ is open and dense in $\partial T$.
Hence there exists $x\in\partial T$ such that $fx\neq x$ for all $f\in F$, so the action of $G$ on $\partial T$ is strongly faithful.

(iii) $\Longrightarrow$ (ii):
Let $F$ be a finite subset of $\Gamma\setminus\{1\}$ and take $x\in\partial T$ such that $fx\neq x$ for all $f\in F$.
By Lemma~\ref{basissets}, there is an edge $e$ in $T$ such that $fZ_B(e)\cap Z_B(e)=\oo$ for all $f\in F$.
For each $f\in F$, the set of $f$-fixed points in $Z_0(e)$ is now a bounded subtree of the unbounded subtree $Z_0(e)$ in $T$. Hence $Z_0(e)$ contains a vertex not fixed by any $f\in F$.

\cite[Corollary~15]{dlhpreaux} ensures the Powers property.
\end{proof}

\begin{theorem}\label{fixator-amenable}
Let $K$ be the a fixator subgroup of a half-tree of $T$ and set $N=\ker(\Gamma\curvearrowright T)$.
The following hold:
\begin{itemize}
\item[(i)] If $\Gamma$ is $C^*$-simple, then $K$ is either trivial or non-amenable;
\item[(ii)] If $K=N$ and $N$ is $C^*$-simple, then $\Gamma$ is $C^*$-simple;
\item[(iii)] If $K/N$ is non-amenable and $N$ is $C^*$-simple, then $\Gamma$ is $C^*$-simple.
\end{itemize}
In particular, if $N$ is trivial, then $\Gamma$ is $C^*$-simple if and only if $K$ is trivial or non-amenable.

Also, if $K$ is amenable, then $\Gamma$ is $C^*$-simple if and only if $K$ is trivial.
\end{theorem}

\begin{proof}
By Lemma~\ref{dontstabmehere}, $K$ is the fixator subgroup $\Gamma_U$ of a non-empty and non-dense open subset $U$ of $\ov{\partial T}$.
Then part~(i) follows from Proposition~\ref{prop:am-free}, while part~(ii) and~(iii) follow from Corollary~\ref{cor:extension},
using that $K/N=\Gamma_U/N=(\Gamma/N)_U$.
\end{proof}

\begin{example}\label{regular-tree}
Let $T$ be a regular tree of degree $d\geq 3$, i.e., all vertices in $T$ have degree $d$,
and let $\Gamma$ be the subgroup of $\operatorname{Aut}(T)$ consisting of all automorphisms that are either elliptic or hyperbolic with an even translation length.
The group $\Gamma$ is also considered in \cite[Remarks~(iv), p.~468]{dlhpreaux}, and it is uncountable and has finite index in $\operatorname{Aut}(T)$.
We will now show that $\Gamma\curvearrowright\ov{\partial T}$ is an extreme boundary action.
As $T$ is regular, $\ov{\partial T}=\partial T$ if $d<\infty$ and $\ov{\partial T}=T\cup \partial T$ if $d=\infty$.

We start by explaining that for every two edges $e$ and $f$ an element $g \in \Gamma$ can be found, such that  $gZ_0(f) \subseteq Z_0(e)$.
The canonical binary partitioning of $T$ induces an orientation of the edges.
There are two cases: $e$ and $f$ have the same orientation, or $e$ and $f$ have opposite orientation.
In the latter case, take an edge $e_1  \in Z_0(e)$ with $r(e) = s(e_1)$.
Then it is easy to see that $e_1$ has the same orientation as $f$ and, of course, $Z_0(e) \supseteq Z_0(e_1)$.
So it is enough to show that there is a group element $g \in \Gamma$ for which $e = gf$ (in the first case), or $e_1 = gf$ (in the second case).
But if $p$ and $q$ are edges of the same orientation, then either the geodesic connecting them has even number of edges, and $\lvert d(s(p),s(q)) - d(r(p),r(q)) \rvert = 2$
(in this case there is an elliptic element $g$ with $p = gq$),
or the geodesic connecting them has odd number of edges, and $\lvert d(s(p),s(q)) - d(r(p),r(q)) \rvert = 0$ (in this case there is a hyperbolic element $g$ with $p = gq$).
But $\Gamma$ contains all elements that preserve the orientation, so $gZ_0(f)=Z_0(gf)\subseteq Z_0(e)$. 

Let $K\subsetneq \ov{\partial T}$ be a closed set and $\oo\neq U\subseteq \ov{\partial T}$ be an open set.
Since $\ov{\partial T}\setminus K$ is open and non-empty,
then by Lemma~\ref{contain-basisset} there exist edges $e$ and $f$ such that $Z(e) \subseteq U$ and $Z(\ov{f}) \subseteq\ov{\partial T}\setminus K$, that is, $K\subseteq Z(f)$.
The element $g$ constructed above, for which $g Z_0(f) \subseteq Z_0(e)$, 
satisfies $g Z(f) \subseteq Z(e)$ by continuity.
Therefore $g K \subseteq g Z(f) \subseteq Z(e) \subseteq U$.
It follows that the actions of $\Gamma$ and $\operatorname{Aut}(T)$ on $\ov{\partial T}$ are extreme boundary actions.

Since $\Gamma$ acts faithfully on $T$, $\Gamma$ has the unique trace property by Proposition~\ref{interior-kernel}.
Moreover, the action is not topologically free (i.e., not slender), so $\operatorname{int}(\Gamma\curvearrowright\partial T)$ is non-trivial,
and thus the simplicity of $\Gamma$ implies that $\operatorname{int}(\Gamma\curvearrowright\partial T)=\Gamma$.
Note that any half-tree of $T$ is a regular rooted tree,
so the fixator subgroup in $\operatorname{Aut}(T)$ of a half-tree is isomorphic to the full automorphism group of a regular rooted tree of branching degree $d$.
This group contains the full automorphism group of a regular rooted tree of branching degree $3$ (i.e., of a rooted binary tree),
which in turn contains a free subgroup on two generators according to \cite{sidkiwilson}.
Note that all automorphisms of a binary rooted tree preserve the orientation, therefore this group is contained in $\Gamma$.
We can conclude that the fixator subgroup in both $\Gamma$ and $\operatorname{Aut}(T)$ of any half-tree is non-amenable.
Therefore, both $\Gamma$ and $\operatorname{Aut}(T)$ are $C^*$-simple by Theorem~\ref{fixator-amenable}.
\end{example}

\begin{remark}
In \cite[Lemma~5.8]{ivanovomland} by the second and third author it was incorrectly stated that $\partial T$ is always a $\Gamma$-boundary,
which was later used the proof of \cite[Theorem~5.9]{ivanovomland}.
The statement of \cite[Theorem~5.9]{ivanovomland} is still correct;
in fact, a much more general result holds, as shown in Theorem~\ref{fixator-amenable}.
\end{remark}

\section{Graphs of groups and HNN extensions}\label{graphs}

In this section we give some of the preliminaries for the branch of geometric group theory now known as \emph{Bass-Serre theory}.
Much of our exposition is based on the original source \cite[{\S I.5}]{serre} for this topic, but is also inspired by the approach taken in \cite{Bass93} and \cite{dlhpreaux}.

\subsection{Graphs of groups}
A \emph{graph of groups} $(G,Y)$ consists of a non-empty connected graph $Y=(V,E,s,r)$,
families of groups $(G_v)_{v\in V}$ and $(G_e)_{e\in E}$ such that $G_{\ov{e}}=G_e$ for all $e\in E$,
and a family of injective group homomorphisms $\varphi_e\colon G_e\to G_{s(e)}$, $e\in E$.
We pick an \emph{orientation} $E_+\subseteq E$ of $Y$,
meaning that $E_+\cap\{e,\ov{e}\}$ contains only one edge for all $e\in E$, and define $E_-=E\setminus E_+$.

Now, let $M=(V(M),E(M))$ be a maximal subtree of $Y$. We define the \emph{fundamental group} $\Gamma=\pi_1(G,Y,M)$ of $(G,Y)$ by
\[
\Gamma=\left\langle \{G_v\}_{v\in V},\ \{\tau_e\}_{e\in E}\,\middle|\,\begin{array}{cl}\tau_{\ov{e}}=\tau_e^{-1}&\text{for all }e\in E,\\
\tau_e\varphi_{\ov{e}}(g)\tau_e^{-1}=\varphi_e(g)&\text{for all }e\in E,\ g\in G_e\\
\tau_e=1&\text{for all }e\in E(M)\end{array}\right\rangle.
\]
We have natural group homomorphisms $G_y\to\Gamma$ for all $y\in V$, and they are all injective.
Moreover, $\tau_e\in\Gamma$ has infinite order for all $e\in E\setminus E(M)$.

We next define a graph $T=(V(T),E(T),s,r)$ as follows.
For any edge $e\in E$, let $\lvert e\rvert$ be the unique edge in $\{e,\ov{e}\}\cap E^+$ and define $\Gamma_e=\varphi_{\lvert e\rvert}(G_e)$.
We define $V(T)$ and $E(T)$ by means of left coset spaces in the following way:
\[
V(T)=\bigsqcup_{v\in V}\Gamma/G_v,\quad E(T)=\bigsqcup_{e\in E}\Gamma/\Gamma_e.
\]
The source, range and inversion maps are given by
\begin{gather*}
s(g\Gamma_e)=\left\{\begin{array}{cl}gG_{s(e)}&\text{for }e\in E^+\\g\tau_e^{-1}G_{s(e)}&\text{for }e\in E^-,\end{array}\right.\quad
r(g\Gamma_e)=\left\{\begin{array}{cl}g\tau_eG_{r(e)}&\text{for }e\in E^+\\gG_{r(e)}&\text{for }e\in E^-,\end{array}\right. \\
\ov{g\Gamma_e}=g\Gamma_{\ov{e}}
\end{gather*}
for all $g\in\Gamma$ and $e\in E$.

\begin{remark}
Note that in \cite{serre} the set-up is slightly different: there $\varphi_e$ is a map $G_e\to G_{r(e)}$, and the relation is $\tau_e\varphi_e(g)\tau_e^{-1}=\varphi_{\ov{e}}(g)$.
However, above we follow the convention used in \cite{Bass93}, which also coincides with the description in \cite[Example~E.13]{brown-ozawa}.
\end{remark}

\begin{theorem}[{\cite[\S I.5]{serre}}]
For $(G,Y)$, $M$ and $E^+$ as above, the graph $T$ constructed above is a tree. Up to isomorphisms, $\Gamma$ and $T$ are independent of the choice of maximal subtree $M$ and orientation $E^+$.
\end{theorem}
The tree $T$ is the so-called \emph{Bass-Serre tree} of the graph of groups $(G,Y)$; we will also say that $T$ is the Bass-Serre tree of the fundamental group $\Gamma$ of $(G,Y)$. 

The definition of $T$ incites us to define an action of $\Gamma$ on $T$ by left translation of cosets, and $\Gamma$ acts on $T$ by automorphisms without inversions,
i.e., $ge\neq\ov{e}$ for all $g\in\Gamma$ and $e\in E(T)$.

\begin{theorem}[Fundamental theorem of Bass-Serre theory]\label{bass-serre}
Suppose that $\Gamma$ is a group acting without inversions on a tree $T$.
Then there exists a graph of groups with fundamental group $\Gamma'$ isomorphic to $\Gamma$ such that the action of $\Gamma'$ on its Bass-Serre tree $T'$ is isomorphic to the action of $\Gamma$ on $T$.
Moreover, the stabilizer subgroups of vertices (resp.\ edges) of $T$ are isomorphic to the vertex (resp.\ edge) groups of the graph of groups under this isomorphism.
\end{theorem}

\begin{remark}
Let $\mathcal G=(G,Y)$ be a graph of groups, and let $\Gamma$ and $T$ be the corresponding fundamental group and Basse-Serre tree, as described above.
In \cite{bmpst}, inspired by the well-known construction of graph $C^*$-algebras,
and under the assumption that $T$ is locally finite and without leaves,
the authors define a $C^*$-algebra $C^*(\mathcal G)$ by generators and relations that are encoded by $\mathcal G$ in a natural way.
Then they prove a $C^*$-algbraic version of Theorem~\ref{bass-serre},
showing in \cite[Theorem~4.1]{bmpst} that $C^*(\mathcal G)$ is isomorphic to the stabilization of the full crossed product $C(\partial T)\rtimes\Gamma$.
One of their goals is to find conditions ensuring that $C^*(\mathcal G)$ is a UCT Kirchberg algebra,
which in our case, i.e., assuming only that $\Gamma\curvearrowright T$ is minimal and strongly hyperbolic, can be characterized completely for the associated crossed products;
namely, we have that the following are equivalent:
\begin{itemize}
\item[(i)] $C(\ov{\partial T})\rtimes\Gamma$ is a UCT Kirchberg algebra,
\item[(ii)] $C(\ov{\partial T})\rtimes_r\Gamma$ is a UCT Kirchberg algebra,
\item[(iii)] $\Gamma$ and $T$ are countable and $\Gamma\curvearrowright\ov{\partial T}$ is amenable and topologically free.
\end{itemize}

(i)~$\Longrightarrow$~(ii):
Nuclearity implies that the action is amenable, so the full and reduced crossed products are isomorphic.

(ii)~$\Longrightarrow$~(iii):
Separability implies that $\Gamma$ and $T$ are countable, and nuclearity implies that $\Gamma\curvearrowright\ov{\partial T}$ is amenable, and thus amenably free by Remark~\ref{rem:amenably}.
Since $C(\ov{\partial T})\rtimes_r\Gamma$ is simple, then $\Gamma$ is $C^*$-simple by \cite[Theorem~6.2]{kalantarkennedy}, so by Proposition~\ref{prop:am-free}, the action must be topologically free.

(iii)~$\Longrightarrow$~(i):
Amenability implies that the full and reduced crossed products are isomorphic, and they are nuclear, while countability implies that they are separable.
Moreover, if a countable group acts amenably on a commutative $C^*$-algebra, then the crossed product staisfies the UCT property (see \cite{tu}).
Finally, topological freeness also implies that $C(\ov{\partial T})\rtimes_r\Gamma$ is simple and purely infinite by \cite[Theorem~5]{lacaspielberg}.
\end{remark}

\begin{remark}
Let $\Gamma$ be a countable discrete group acting on a countable tree $T$.
Then, following the argument given in \cite[Theorem~5.29]{bmpst}, applying \cite[Theorem~5.2.1 and Lemma~5.2.6]{brown-ozawa}, and using Remark~\ref{rem:amenably},
the following are equivalent:
\begin{itemize}
\item[(i)] $\Gamma_x$ is amenable for all $x\in T$,
\item[(ii)] $\Gamma\curvearrowright T\cup\partial T$ is amenable,
\item[(iii)] $\Gamma_x$ is amenable for all $x\in T\cup\partial T$.
\end{itemize}
If any of these equivalent conditions hold, then $\Gamma\curvearrowright\overline{\partial T}$ is amenable.
\end{remark}

We complete this subsection by discussing a few well-known normal subgroups.
For a group $G$, let $FC(G)$ denote the FC-center of $G$, i.e., normal subgroup of elements in $G$ of finite conjugacy class,
and let $NF(G)$ denote the largest normal subgroup of $G$ that does not contain any non-abelian free subgroup.
Now let $\Gamma$ be a fundamental group of graphs of groups acting on its Bass-Serre tree $T$ and just write $\ker{\Gamma}$ for $\ker(\Gamma\curvearrowright T)$.
As usual, we assume that the action is minimal and strongly hyperbolic.
Then we have the following sequence of subgroups:
\begin{equation}\label{eq:sequence}
FC(\Gamma) \subseteq FC(\ker{\Gamma}) \subseteq R(\ker{\Gamma}) = R(\Gamma) \subseteq NF(\Gamma) = NF(\ker{\Gamma}) \subseteq \ker{\Gamma}.
\end{equation}
The inclusion $R(\Gamma) \subseteq NF(\Gamma)$ is well-known to hold for any group \cite{day1957},
and the two equalities follow from \cite[Examples~6.4, 6.6, and Lemma~6.7]{ivanovomland}.
The last inclusion follows from \cite[Proposition~19]{cornulier}.
Note that this implies that $\Gamma$ always contains a non-abelian free subgroup.
The first containment is clear since $FC(\Gamma)$ is a subgroup of $\ker{\Gamma}$ and the second holds since $FC(\ker{\Gamma})$ is an amenable normal subgroup of $\ker{\Gamma}$.
In general, $FC(\ker{\Gamma})$ is always normal in $\Gamma$ (since an FC-center is a characteristic subgroup), but it may be bigger than $FC(\Gamma)$.

Moreover, if $\ker{\Gamma}$ is finite, then $FC(\Gamma)=\ker{\Gamma}$, and the sequence \eqref{eq:sequence} collapses.
To see why, note that $h\in\ker{\Gamma}=g^{-1}(\ker{\Gamma})g$ for all $g\in\Gamma$, so $ghg^{-1}\in\ker{\Gamma}$ for all $g\in\Gamma$.
Therefore the conjugacy class of $h$ is contained in $\ker{\Gamma}$, which is finite, so $h\in FC(\Gamma)$.

\subsection{HNN extensions}\label{HNNbasics}

We now zoom in on the case of HNN extensions, first from an algebraic point of view,
for which we provide some structure results and define subgroups which we shall interpret geometrically in the next subsection.
Our work should be compared with similar results by the second and third author for free products with amalgamation \cite{ivanovomland}:
whereas HNN extensions are fundamental groups of loops, free amalgamated products are fundamental groups of a segment of length $1$.

Suppose that $(F,Y)$ is a graph of groups, where $Y$ is a loop with one vertex $v$ and one pair of edges $\{e,\ov{e}\}$.
We let $\{e\}$ be the orientation of $Y$ and $G=F_v$.
As the homomorphisms $\varphi_e\colon F_e\to G$ and $\varphi_{\ov{e}}\colon F_e\to G$ of $F_e$ are both injective,
we may define $H=\varphi_e(F_e)$ and an injective group homomorphism $\theta=\varphi_{\ov{e}}\varphi_e^{-1}\colon H\to G$.
Defining the \emph{stable letter} $\tau=\tau_e$, the fundamental group $\Gamma$ is the well-known \emph{HNN extension} $\operatorname{HNN}(G,H,\theta)$:
\[
\Gamma=\operatorname{HNN}(G,H,\theta)=\langle G, \tau\mid\tau^{-1} h\tau=\theta(h)\text{ for all }h\in H\rangle.
\]
These groups are named after Higman, Neumann, and Neumann, who first studied them in \cite{hnn}.

Let $\Gamma=\operatorname{HNN}(G,H,\theta)$ be an HNN extension and let $T$ be the Bass-Serre tree of $\Gamma$.
The vertices in $T$ are left cosets of $\Gamma/G$ and the edge set of $T$ consists of two disjoint copies of $\Gamma/H$, say, $\Gamma/H\sqcup\ov{\Gamma/H}$,
where the inversion map sends $gH$ to $\ov{gH}$ and vice versa. The source and range maps are given by
\[
s(gH)=gG=r(\ov{gH}),\quad  r(gH)=g\tau G=s(\ov{gH}),\quad g\in\Gamma.
\]
The action of $\Gamma$ on $T$ by left translation is transitive. Moreover, $T$ is regular, i.e., all vertices have the same degree $[G:H]+[G:\theta(H)]$,
so that in particular, $T$ is leafless, and $T$ is countable if and only if $G/H$ and $G/\theta(H)$ are of at most countably infinite cardinality.

The picture below illustrates part of the Bass-Serre tree of the HNN extension of a group $G$, and subgroups $H\cong \theta(H)$ such that $[G:H]=2$ and $[G:\theta(H)]=3$.
We let $S_{-1}=\{1,s\}$ and $S_1=\{1,t_1,t_2\}$ be sets of left coset representatives for $H$ and $\theta(H)$, respectively.
Observe that for $g\in S_{-1}$, if we want to add $g\tau$ to the right in one of the vertex cosets $mG$ (e.g., going from $mG$ to $mg\tau G$),
we traverse an edge \emph{emanating} from $mG$, and when adding $g\tau^{-1}$ to the right for $g\in S_1$, one traverses an edge \emph{ending} in $mG$.

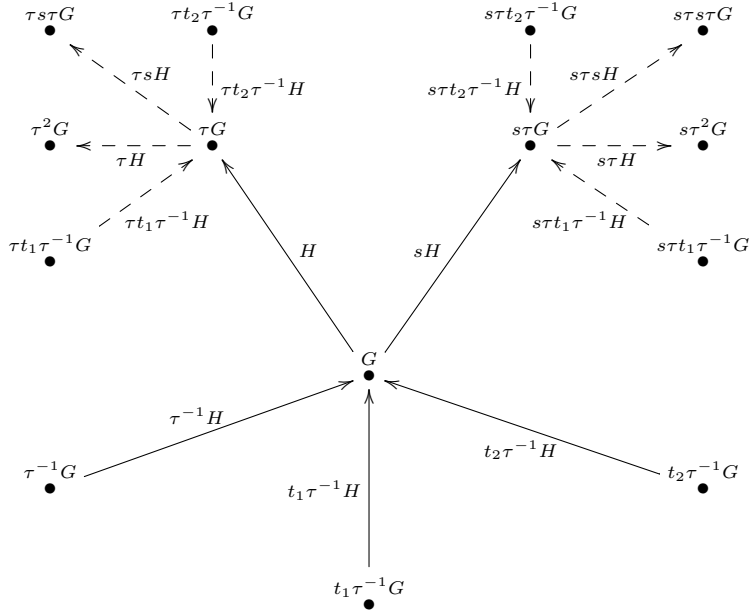
\[
\xymatrix{
&\stackrel{\tau s\tau G}{\bullet}&\stackrel{\tau t_2\tau^{-1} G}{\bullet}\ar@{-->}[d]^{\tau t_2\tau^{-1} H}&&\stackrel{s\tau t_2\tau^{-1} G}{\bullet}\ar@{-->}[d]_{s\tau t_2\tau^{-1} H}&\stackrel{s\tau s\tau G}{\bullet}&\\
&\stackrel{\tau^2 G}{\bullet}&\stackrel{\tau G}{\bullet}\ar@{-->}[l]^{\tau H}\ar@{-->}[ul]_{\tau sH}&&\stackrel{s\tau G}{\bullet}\ar@{-->}[r]_{s\tau H}\ar@{-->}[ur]^{s\tau sH}&\stackrel{s\tau^2 G}{\bullet}&\\
&\stackrel{\tau t_1\tau^{-1} G}{\bullet}\ar@{-->}[ur]_{\tau t_1\tau^{-1} H}	&&&&\stackrel{s\tau t_1\tau^{-1} G}{\bullet}\ar@{-->}[ul]^{s\tau t_1\tau^{-1} H}&\\
&&&\stackrel{G}{\bullet} \ar[uul]_{H}\ar[uur]^{sH}&&&\\
&\stackrel{\tau^{-1} G}{\bullet}\ar[urr]^{\tau^{-1} H}&&&&\stackrel{t_2\tau^{-1} G}{\bullet}\ar[ull]^{t_2\tau^{-1} H}&\\
&&&\stackrel{t_1\tau^{-1} G}{\bullet}\ar[uu]^{t_1\tau^{-1} H}&&&}
\]

Let $g\in\Gamma$ be an element with normal form $g=g_1\tau^{\eps_1}g_2\tau^{\eps_2}\cdots g_n\tau^{\eps_n}$.
Defining $t_0=1$ and
\begin{equation}\label{funkydiag}
t_k=\prod_{i=1}^kg_i\tau^{\eps_i},\quad e_k=\left\{\begin{array}{cl} t_k\tau^{-1} H =t_{k-1}g_kH&\text{if }\eps_k=1,\\\ov{t_kH}&\text{if }\eps_k=-1\end{array}\right.
\end{equation}
for $1\leq k\leq n$, the unique path from $G$ to $gG$ is given by 
\[
\xymatrix{G\ar[rr]^{e_1}&&t_1G\ar[rr]^{e_2}&&\cdots\ar[rr]^{e_n}&&gG}
\]
Indeed, if $\eps_k=1$, then $s(t_{k-1}g_kH)=t_{k-1}G$ and $r(t_{k-1}g_kH)=t_{k-1}g_k\tau G=t_kG$, and if $\eps_k=-1$ we notice that $s(\ov{t_kH})=t_k\tau G=t_{k-1}G$ and $r(\ov{t_kH})=t_kG$.

If we let $S_{-1}$ and $S_1$ be systems of representatives for the \emph{left} cosets of $H$ and $\theta(H)$ in $G$, respectively, such that $1\in S_{-1}\cap S_1$,
the \emph{unique normal form} of an element $g\in\Gamma$ (see, e.g., \cite[Theorem~2.14.3]{bogopolski}) can be written
\[
g=g_1\tau^{\eps_1}g_2\tau^{\eps_2}\cdots g_n\tau^{\eps_n}g_{n+1},
\]
where $n\geq 0$, and the following conditions are satisfied:
\begin{itemize}
\item[(i)] $g_{n+1}\in G$ and $\eps_i\in\{\pm1\}$ for $1\leq i\leq n$;
\item[(ii)] $g_i\in S_{-\eps_i}$ for all $1\leq i\leq n$;
\item[(iii)] $g_i=e$ implies $\eps_{i-1}=\eps_i$ for $2\leq i\leq n$.
\end{itemize}
This also entails that the natural map $G\to\Gamma$ is actually an injection.
With $g$ as above, we say that $n=\lvert g\rvert$ is the \emph{length} of $g$, and if $n\geq 1$, we say that $\eps_1$ and $\eps_n$ is the \emph{type} and \emph{direction} of $g$, respectively.
The \emph{initial letter} of $g$ is $g_1\in S_{-\eps_1}$ and the \emph{end letter} of $g$ is $g_{n+1}\in G$. We define the length of any $g\in G$ to be $0$,
and the initial letter and end letter of $g$ as an element in $\Gamma$ are given by $1$ and $g$, respectively.

\medskip
For $n\geq 1$, $g_1,\ldots,g_{n+1}\in G$ and $\eps_1,\ldots,\eps_n\in\{\pm1\}$, the word
\[
g_1\tau^{\eps_1}\cdots g_n\tau^{\eps_n}g_{n+1}
\]
is said to be \emph{reduced} (in $\Gamma$) if for all $1\leq i\leq n-1$ we have
\begin{itemize}
\item[(a)] $g_{i+1}\notin H$ whenever $\eps_i=-1$ and $\eps_{i+1}=1$, and
\item[(b)] $g_{i+1}\notin\theta(H)$ whenever $\eps_i=1$ and $\eps_{i+1}=-1$.
\end{itemize}
If we define
\[
H_{-1}=H,\quad H_1=\theta(H),
\]
the conditions (a) and (b) can be rephrased as follows: for all $1\leq i\leq n-1$, $g_{i+1}\notin H_{\eps_i}$ whenever $\eps_{i+1}=-\eps_i$.
Notice that
\begin{equation}\label{funnypower}
\tau^{-\eps}H_{-\eps}\tau^{\eps}=H_{\eps}.
\end{equation}
We say that $g_1\in G$ is \emph{reduced} if $g_1\neq 1$.
A fundamental result for HNN extensions, also known as \emph{Britton's lemma}, is that reduced words always define non-identity elements.
The result itself can be derived from the uniqueness of the normal form.
Indeed, if for $n\geq 1$ the word $g=g_1\tau^{\eps_1}\cdots g_n\tau^{\eps_n}g_{n+1}\in\Gamma$ is reduced, let $s_1\in S_{-\eps_1}$ and $h_1\in H_{-\eps_1}$ such that $g_1=s_1h_1$ and rewrite
\[
g=s_1\tau^{\eps_1}(\tau^{-\eps_1}h_1\tau^{\eps_1})g_2\tau^{\eps_2}\cdots g_n\tau^{\eps_n}g_{n+1}.
\]
The remainder of the proof divides into two possible situations, depending on whether consecutive powers of $\tau$ in the word coincide.
Indeed, define $g_2'=(\tau^{-\eps_1}h_1\tau^{\eps_1})g_2$.
If $\eps_2=\eps_1$, then write $g_2'=s_2h_2$ for $s_2\in S_{-\eps_2}$ and $h_2\in H_{-\eps_2}$ as above, and write
\[
g=s_1\tau^{\eps_1}g_2'\tau^{\eps_2}\cdots g_n\tau^{\eps_n}g_{n+1}=s_1\tau^{\eps_1}s_2\tau^{\eps_2}(\tau^{-\eps_2}h_2\tau^{\eps_2})g_3\cdots g_n\tau^{\eps_n}g_{n+1}.
\]
If $\eps_2=-\eps_1$, then $g_2\notin H_{-\eps_2}=H_{\eps_1}$ by assumption, so that due to (\ref{funnypower}), $g_2'=(\tau^{-\eps_1}h_1\tau^{\eps_1})g_2\in H_{\eps_1}g_2$ and $g_2'\notin H_{\eps_1}=H_{-\eps_2}$.
We then proceed as for $g_1$, noting that the coset representative of $g_2'$ with respect to $H_{-\eps_2}$ is not $1$.
Iterating the process yields the normal form of $g$, which contains at least $2n-1$ terms. If $n\geq 2$ then $g\neq 1$ due to uniqueness of the normal form, and if $n=1$,
then $g=g_1\tau^{\eps_1}g_2\neq 1$ would imply $\tau\in G$, a contradiction.
\medskip

The above proof of Britton's lemma also proves the following lemma.
\begin{lemma}\label{typeshit}
Let $n\geq 1$ and let $g=g_1\tau^{\eps_1}\cdots g_n\tau^{\eps_n}g_{n+1}\in\Gamma$ be a reduced word.
Then
\begin{itemize}
\item[(i)] $g\notin G$ and $\lvert g\rvert=n$;
\item[(ii)] the type of $g$ is $\eps_1$;
\item[(iii)] the direction of $g$ is $\eps_n$;
\item[(iv)] the initial letter of $g$ is $1$ if and only if $g_1\in H_{-\eps_1}$;
\item[(v)] if $g_{n+1}\in H_{\eps_n}$, the end letter of the normal form is contained in $H_{\eps_n}$ as well.
\end{itemize}
\end{lemma}

In some cases when working with large subsets of an HNN extension, it will prove helpful to be able to uncover properties of elements in these subsets without having to reduce.
We introduce a simple lemma to remedy this situation, the proof modeled after the proof of the preceding lemma.

\begin{lemma}\label{typeshit2}
For an HNN extension $\Gamma=\operatorname{HNN}(G,H,\theta)$, $n\geq 1$, $g_1,\ldots,g_{n+1}\in G$ and $\eps_1,\ldots,\eps_n\in\{\pm1\}$, define
\[
g=g_1\tau^{\eps_1}\cdots g_n\tau^{\eps_n}g_{n+1}\in\Gamma.
\]
If $n$ is odd, then $g\notin G$, and if $\eps_1=\ldots=\eps_k$ for some $k>\frac{n}{2}$, the type of $g$ is $\eps_1$.
\end{lemma}
\begin{proof}
Notice that the above expression of $g$ is reduced if and only if $\tau^{\eps_j}g_{j+1}\tau^{\eps_{j+1}}\notin G$ for all $1\leq j\leq n-1$.
We may therefore assume that there is $1\leq j\leq n-1$ such that $\tau^{\eps_j}g_{j+1}\tau^{\eps_{j+1}}\in G$. Let $1\leq i\leq\min\{j,n-j\}$ be largest such that
\[
h=g_{j-(i-1)}\tau^{\eps_{j-(i-1)}}g_{j-(i-1)+1}\cdots g_{j+i}\tau^{\eps_{j+i}}g_{j+i+1}\in G.
\]
We may now write $g=g_1\tau^{\eps_1}\cdots g_{j-i}\tau^{\eps_{j-i}}h\tau^{\eps_{j+i+1}}\cdots g_n\tau^{\eps_n}g_{n+1}$.
If this word is not reduced, we can continue this process for this new expression of $g$.
After a finite number of iterations, the word must be reduced,
so because this process always removes an even number of powers of $\tau$ from the preceding expression of $g$, $n$ being odd implies that $g\notin G$ by Lemma~\ref{typeshit}~(i).
If $\eps_1=\ldots=\eps_k$ for some $k>\frac{n}{2}$, then reduction removes at most $n-k<k$ of the identical first $k$ powers of $\tau$ in the original expression of $g$.
Therefore the type of $g$ is $\eps_1$ by Lemma~\ref{typeshit}~(ii).
\end{proof}

The \emph{kernel} of the HNN extension $\Gamma$ is the normal subgroup
\[
\ker{\Gamma}=\bigcap_{r\in\Gamma}rHr^{-1}.
\]
For $\eps\in\{\pm1\}$, let $T_\eps$ be the collection of elements $g\in\Gamma$ of length $n\geq 1$ and type $\eps$.
Let $T_\eps^\dagger$ be the subset of $g\in T_\eps$ of length $n\geq 1$ with initial letter $1$.
We then define the \emph{quasi-kernel}
\[
K_\eps=\bigcap_{r\in\Gamma\setminus T_\eps^\dagger}rHr^{-1}.
\]
Notice that $(\Gamma\setminus T_{-1}^\dagger)\cup(\Gamma\setminus T_1^\dagger)=\Gamma$, so that $\ker{\Gamma}=K_{-1}\cap K_1$.

\medskip

We will consider criteria for HNN extensions to be $C^*$-simple and to have the unique trace property in the following.
An HNN extension $\Gamma=\operatorname{HNN}(G,H,\theta)$ is \emph{ascending} if either $H=G$ or $\theta(H)=G$.
In order to make the most of Britton's lemma, we will mostly consider \emph{non-ascending} HNN extensions, which is not too restrictive of a property,
meaning that both $S_1$ and $S_{-1}$ are non-trivial.


\begin{lemma}\label{wordlemma}
If $\Gamma$ is a non-ascending HNN extension, the normal closures of the quasi-kernels in $\Gamma$ coincide.
\end{lemma}
\begin{proof}
For any $s\in S_{-1}\setminus\{1\}$, $s\tau(\Gamma\setminus T_{-1}^\dagger)\subseteq\Gamma\setminus T_1^\dagger$, so that
\[
K_1=\bigcap_{r\in\Gamma\setminus T_1^\dagger}rHr^{-1}\subseteq s\tau\left(\bigcap_{r\in\Gamma\setminus T_{-1}^\dagger}rHr^{-1}\right)\tau^{-1}s^{-1}=s\tau K_{-1}\tau^{-1}s^{-1}.
\]
Hence the normal closure of $K_{-1}$ contains that of $K_1$. The reverse inclusion is seen in a similar manner.
\end{proof}

As the normal closures of $K_{1}$ and $K_{-1}$ in $\Gamma$ coincide by the above lemma, it follows that $K_{1}$ is trivial if and only if $K_{-1}$ is trivial.
Moreover, if one of $K_{1}$ and $K_{-1}$ is equal to the normal subgroup $\ker{\Gamma}$, then the other one is also equal to $\ker{\Gamma}$.

For a non-ascending HNN extension $\Gamma$, let $\operatorname{int}{\Gamma}$ be the normal closure of either of the quasi-kernels.
We call $\operatorname{int}{\Gamma}$ the \emph{interior} of $\Gamma$.


\begin{remark}
Let $\Gamma=\operatorname{HNN}(G,H,\theta)$ be a non-ascending HNN extension.
Then $\operatorname{int}{\Gamma}=\ker{\Gamma}$ if and only if the quasi-kernels of $\Gamma$ coincide with the kernel of $\Gamma$,
if and only if there exists $g_{\eps}\in G\setminus H_{-\eps}$ such that $g_{\eps}K_{\eps}g_{\eps}^{-1}=K_{\eps}$ for $\eps\in\{\pm1\}$.
To see this, note first that $\eps\in\{\pm1\}$ and all $g_{\eps}\in G\setminus H_{-\eps}$ we have
\[
\Gamma\setminus T_{-\eps}^\dagger\subseteq (\Gamma\setminus T_{\eps}^\dagger)\cup g_{\eps}(\Gamma\setminus T_{\eps}^\dagger).
\]
Indeed, if $x\in T^{\dagger}_\eps\subseteq\Gamma\setminus T^{\dagger}_{-\eps}$, then $g_{\eps}^{-1}x$ has type $\eps$ and initial letter different from $1$.
Therefore $\Gamma=(\Gamma\setminus T_{\eps}^\dagger)\cup g_{\eps}(\Gamma\setminus T_{\eps}^\dagger)$, so that
\[
\ker{\Gamma}=\bigcap_{r\in\Gamma\setminus T_{\eps}^\dagger}rHr^{-1}\cap\bigcap_{r\in\Gamma\setminus T_{\eps}^\dagger}g_{\eps}rHr^{-1}g_{\eps}^{-1}=K_{\eps}\cap g_{\eps}K_{\eps}g_{\eps}^{-1}.
\]
The conclusions now follow from normality.
\end{remark}

The following theorem (and the main result of this section) is motivated by condition (SF') in \cite[Proposition~11]{delaharpe1985} and \cite[Theorem~3.2]{ivanovomland}.

\begin{theorem}\label{hnnpowers}
Let $\Gamma=\operatorname{HNN}(G,H,\theta)$ be a non-ascending HNN extension. The following are equivalent:
\begin{itemize}
\item[(i)] $\operatorname{int}{\Gamma}=\{1\}$;
\item[(ii)] the quasi-kernel $K_\eps$ is trivial for some or both $\eps$;
\item[(iii)] for every finite subset $F\subseteq \Gamma\setminus\{1\}$, there exists $g\in \Gamma$ such that $gFg^{-1}\cap G=\oo$;
\item[(iv)] for every finite subset $F\subseteq \Gamma\setminus\{1\}$, there exists $g\in \Gamma$ such that $gFg^{-1}\cap H=\oo$;
\item[(v)] for every finite subset $F\subseteq G\setminus\{1\}$, there exists $g\in \Gamma$ such that $gFg^{-1}\cap H=\oo$;
\item[(vi)] for every finite subset $F\subseteq H\setminus\{1\}$, there exists $g\in \Gamma$ such that $gFg^{-1}\cap H=\oo$.
\end{itemize}
Moreover, if these conditions hold, then $\Gamma$ is a Powers group, and in particular $C^*$-simple.
\end{theorem}

\begin{proof}
(iii) $\Longrightarrow$ (iv) $\Longrightarrow$ (v) $\Longrightarrow$ (vi) and (i) $\Longleftrightarrow$ (ii) are obvious.

(vi) $\Longrightarrow$ (ii): If $K_\eps$ is trivial for some $\eps\in\{\pm1\}$, then $K_{-\eps}$ is trivial by Lemma~\ref{wordlemma}.
Therefore suppose that $K_{-1}$ and $K_1$ are both non-trivial.
Pick $f_\eps\in K_\eps\setminus\{1\}$ for $\eps\in\{\pm1\}$ and set $F=\{f_{-1},f_1\}$.
For an arbitrary $g\in \Gamma$, we have $g^{-1}\in \Gamma\setminus T_\eps^\dagger$ for some $\eps$.
Then $gf_\eps g^{-1}\in H$, i.e., $gf_\eps g^{-1}\in gFg^{-1} \cap H$.



(ii) $\Longrightarrow$ (iii): Choose a finite $F\subseteq \Gamma\setminus\{1\}$.
Assume first there is an element $f_1\in F \cap G$ (otherwise, there is nothing to show).
Since $f_1\neq 1$, we may pick $g_1\in \Gamma\setminus T_1^\dagger$ such that $g_1^{-1}f_1g_1\notin H$.
We may now assume that $g_1^{-1}f_1g_1\notin G$; if $g_1^{-1}f_1g_1\in G$, we can freely replace $g_1$ by $g_1\tau$.
In particular, $g_1\notin G$, and so we may let $\eps_1$ be the direction of $g_1$.
We then see that the type and direction of $g_1^{-1}f_1g_1$ are $-\eps_1$ and $\eps_1$, respectively,
since we can write $g_1^{-1}f_1g_1$ as a reduced word by means of the normal form of $g_1$ and then apply Lemma~\ref{typeshit}.
In this way, we also see that replacing $g_1$ by $g_1h^{-1}$ where $h$ is the end letter of $g_1$ does not change this conclusion,
so we may assume that $g_1^{-1}f_1g_1$ has initial letter $1$ and end letter contained in $H_{\eps_1}$.

We now assume that there is an element $f_2\in F$ such that $g_1^{-1}f_2g_1\in G$ (otherwise, we are done).
Pick $g_2\in \Gamma\setminus T_{-\eps_1}^\dagger$ such that $g_2^{-1}g_1^{-1}f_2g_1g_2\notin H$.
In the same manner as above, we may assume that $g_2^{-1}g_1^{-1}f_2g_1g_2\notin G$, $g_2\notin G$ and $g_2$ has end letter $1$, and if $\eps_2$ is the direction of $g_2$,
then $g_2^{-1}g_1^{-1}f_2g_1g_2$ has type $-\eps_2$ and direction $\eps_2$.
We now claim that $g_2^{-1}g_1^{-1}f_1g_1g_2\notin G$ as well. Indeed, since $g_1^{-1}f_1g_1\notin G$ has type $-\eps_1$ and direction $\eps_1$ with initial letter $1$ and end letter in $H_{\eps_1}$,
then $g_2^{-1}g_1^{-1}f_1g_1g_2$ has type $-\eps_2$ and direction $\eps_2$, with initial letter $1$ and end letter contained in $H_{\eps_2}$.
To realize this, we consider the normal forms of $g_1^{-1}f_1g_1$ and of $g_2$, say, $g_2=h_1\tau^{\eps}\cdots h_m\tau^{\eps_2}$.
Then
\[
g_2^{-1}(g_1^{-1}f_1g_1)g_2=\tau^{-\eps_2} h_m^{-1}\cdots\tau^{-\eps}h_1^{-1}(g_1^{-1}f_1g_1)h_1\tau^{\eps}\cdots h_m\tau^{\eps_2}.
\]
Hence reduction is only possible if $\eps=-\eps_1$, but $h_1\notin H_{-\eps}=H_{\eps_1}$ by assumption since $g_2\in\Gamma\setminus T_{-\eps_1}^\dagger$.
Therefore the above word is always reduced, so Lemma~\ref{typeshit} applies.

It should be clear how this process continues, choosing $g_i$ from the set $\Gamma\setminus T_\eps^\dagger$, depending on the direction of $g_{i-1}$,
and since $F$ is finite, we take $g$ to be the product of the $g_i$'s,
and then $g^{-1}fg\notin G$ for every $f\in F$.

We refer to \cite[Proposition~11]{delaharpe1985} for a proof that any of the six conditions implies that $\Gamma$ is a Powers group, if $\Gamma$ is non-ascending.
\end{proof}

\begin{remark}\label{delaharpecrit}
In the 2011 article \cite[Theorem~3~(ii)]{dlhpreaux}, a sufficient criterion to ensure $C^*$-simplicity of a non-ascending, countable HNN extension was given by de~la~Harpe and Pr\'{e}aux, formulated as follows.
For $\Gamma=\operatorname{HNN}(G,H,\theta)$ non-ascending and $G$ countable, define $H_0=H$,
and recursively define a descending sequence of subgroups $(H_k)_{k\geq 1}$ of $H_0=H$ by $H_k'=H_k\cap\tau^{-1}H_k\tau$ and
\[
H_k=\left(\bigcap_{g\in G}gH_k'g^{-1}\right)\cap\tau\left(\bigcap_{g\in G}gH_k'g^{-1}\right)\tau^{-1}.
\]
The criterion ensuring $C^*$-simplicity of $\Gamma$ was that $H_k=\{1\}$ for some $k\geq 0$ (in fact, $\Gamma$ is a Powers group).

We claim that Theorem~\ref{hnnpowers} is a stronger result.
Indeed, for $k\geq 1$, let $C_k$ be the set of elements in $\Gamma$ of length $\leq k+1$.
Then $H_k=\{1\}$ implies $\bigcap_{r\in C_k}rHr^{-1}=\{1\}$, since each $H_k$ is obtained by taking intersections of sets of the form $rHr^{-1}$, $r$ running through a subset of $C_k$.
For $\eps\in\{\pm1\}$ and $s\in S_{-\eps}\setminus\{1\}$, then $s\tau^{(k+2)\eps}C_k\subseteq\Gamma\setminus T^\dagger_{\eps}$ due to Lemma~\ref{typeshit2}.
Therefore
\[
K_\eps=\bigcap_{r\in\Gamma\setminus T^\dagger_\eps}rHr^{-1}\subseteq s\tau^{(k+2)\eps}\left(\bigcap_{r\in C_k}rHr^{-1}\right)\tau^{-(k+2)\eps}s^{-1}=\{1\}.
\]
\end{remark}

The following result, similar to \cite[Theorem~3.7]{ivanovomland}, now holds.
\begin{proposition}\label{hnn-finite}
Let $\Gamma=\operatorname{HNN}(G,H,\theta)$ be a non-ascending HNN extension such that $H\cap\theta(H)$ is finite. Then the following are equivalent:
\begin{itemize}
\item[(i)] $\Gamma$ is icc;
\item[(ii)] $\ker{\Gamma}=\{1\}$;
\item[(iii)] $\operatorname{int}{\Gamma}=\{1\}$;
\item[(iv)] $\Gamma$ is a Powers group;
\item[(v)] $\Gamma$ is $C^*$-simple;
\item[(vi)] $\Gamma$ has the unique trace property.
\end{itemize}
\end{proposition}

\begin{proof}
It is clear that (iii)~$\Longrightarrow$~(iv)~$\Longrightarrow$~(v)~$\Longrightarrow$~(vi)~$\Longrightarrow$~(i).

Note that $\ker{\Gamma}\subseteq H\cap\theta(H)$ so the kernel is finite, and therefore (i)~$\Longrightarrow$~(ii) by \eqref{eq:sequence} and the preceding comment.

Thus only (ii)$\,\Longrightarrow\,$(iii) remains. Suppose that $\ker{\Gamma}=\{1\}$.
Note that $\ker{\Gamma}$ coincides with the intersection of the decreasing sequence
\[
H_0\supseteq H_1\supseteq \dotsb H_k\supseteq H_{k+1}\supseteq \dotsb
\]
of Remark~\ref{delaharpecrit}.
For $k\geq 1$, $H_k$ is a subgroup of $H\cap\theta(H)$ and therefore finite, meaning that $H_k$ must be trivial for some $k\geq 1$.
As in Remark~\ref{delaharpecrit}, we conclude that $K_\eps$ is trivial for $\eps\in\{\pm 1\}$.
\end{proof}
The above result indicates that in order to search for examples of non-ascending HNN extensions $\operatorname{HNN}(G,H,\theta)$ that are not $C^*$-simple but do satisfy the unique trace property,
we have to ensure that the image of $\theta$ inside $G$ is not too far away from $H$.
We give such an example in Section~\ref{example}.

\medskip

An alternate characterization of recurrence of a subgroup $H$ of a group $G$
is that there exists a finite subset $F\subseteq G\setminus\{1\}$ such that $F\cap gHg^{-1}\neq\oo$ for all $g\in G$ (see \cite[Section~5]{kennedy2015char}).
Combining \cite[Proposition~5.2]{kennedy2015char} with Theorem~\ref{hnnpowers}, we obtain the following result.

\begin{lemma}\label{recurrence}
Let $\Gamma=\operatorname{HNN}(G,H,\theta)$ be a non-ascending HNN extension. The following are equivalent:
\begin{itemize}
\item[(i)] $K_{-1}$ and $K_1$ are non-trivial;
\item[(ii)] $H$ is recurrent in $\Gamma$;
\item[(iii)] $G$ is recurrent in $\Gamma$.
\end{itemize}
\end{lemma}

\begin{corollary}
Let $\Gamma=\operatorname{HNN}(G,H,\theta)$ be a non-ascending HNN extension.
If $H$ is non-recurrent, then $\Gamma$ is $C^*$-simple.
Consequently, if $H$ is amenable, then $\Gamma$ is $C^*$-simple if and only if $H$ is \emph{not} recurrent in $\Gamma$.
\end{corollary}

\begin{proof}
If $H$ is not recurrent, then $K_{-1}$ and $K_1$ are trivial by Lemma~\ref{recurrence}, and therefore $\Gamma$ is $C^*$-simple by Theorem~\ref{hnnpowers}.
The last part now follows from \cite[Theorem~5.3]{kennedy2015char}.
\end{proof}

\begin{remark}
There is also another way of seeing that $\Gamma=\operatorname{HNN}(G,H,\theta)$ is not $C^*$-simple when $H$ is amenable and the quasi-kernels $K_{-1}$ and $K_1$ are non-trivial.

Note first that if $K_1=\ker{\Gamma}$, then $K_{-1}=\ker{\Gamma}$.
Thus, by assumption, $\ker{\Gamma}$ is a non-trivial normal amenable subgroup of $\Gamma$, and hence $\Gamma$ cannot be $C^*$-simple.
Hence, we may assume that both $K_1$ and $K_{-1}$ are different from $\ker{\Gamma}$.

Choose $a\in K_1\setminus\ker{\Gamma}$ and $b\in K_{-1}\setminus\ker{\Gamma}$.
Then
\[
\begin{split}
\{gH \mid  gH\neq agH\}&\subseteq\{gH\mid g\in T_{-1}\} \quad\text{and}\\
\{gH \mid  gH\neq bgH\}&\subseteq\{gH\mid g\in T_1\},
\end{split}
\]
which are clearly disjoint.
By using the technique from Proposition~4.8 in \cite{haagerupolesen}, also explained in \cite[p.~12]{ozawahandout},
the action of $\Gamma$ on $\Gamma/H$ gives rise to a unitary representation $\pi\colon G\to\ell^2(\Gamma/H)$,
that extends to a continuous representation of $C^*_r(\Gamma)$.
It follows that $(1-\lambda(a))(1-\lambda(b))$ generates a proper two-sided closed ideal of $C^*_r(\Gamma)$.
Hence, $\Gamma$ is not $C^*$-simple.
\end{remark}

\subsection{Boundary actions of non-ascending HNN extensions}

Now let $T$ denote the Bass-Serre tree of an HNN extension $\Gamma=\operatorname{HNN}(G,H,\theta)$.
It was remarked in the beginning of the previous subsection that the action of $\Gamma$ on $T$ is transitive, so it is also minimal.
Moreover, we either have $\ov{\partial T}=\partial T$ or $\ov{\partial T}=T\cup\partial T$ in $T\cup\partial T$ in the shadow topology, since $T$ is regular.
If $\lvert\partial T\rvert\leq 2$, every vertex in $T$ has degree~$2$, so that $H=\theta(H)=G$ and $\Gamma=G\rtimes_{\theta}\Z$.

\begin{proposition}
Let $\Gamma=\operatorname{HNN}(G,H,\theta)$ be an HNN extension, and let $T$ be the Bass-Serre tree of $\Gamma$. If $\lvert\partial T\rvert\geq 3$, then the following are equivalent:
\begin{itemize}
\item[(i)] The action of $\Gamma$ on $T$ is strongly hyperbolic.
\item[(ii)] $\Gamma$ is non-ascending.
\end{itemize}
If any of these two conditions is satisfied, $\ov{\partial T}$ is a $\Gamma$-boundary in the shadow topology.
\end{proposition}

\begin{proof}
The equivalence follows from \cite[Proposition~20]{dlhpreaux}, and since the action of $\Gamma$ on $T$ is minimal, it follows from Lemma~\ref{lem:T-boundary} that $\ov{\partial T}$ is indeed a $\Gamma$-boundary.
\end{proof}

The shadows in $T\cup\partial T$ can be described as follows.
If $g$ has normal form
\[
g=g_1\tau^{\eps_1}g_2\tau^{\eps_2}\cdots g_n\tau^{\eps_n}g_{n+1}
\]
for $n\geq 1$, let $U(g)$ be the subset of all elements $hG$ where the normal form of $h$ begins with $gg_{n+1}^{-1}=g_1\tau^{\eps_1}g_2\tau^{\eps_2}\cdots g_n\tau^{\eps_n}$,
as well as all equivalence classes of rays identifiable with a ray beginning with $g_1\tau^{\eps_1}g_2\tau^{\eps_2}\cdots g_n\tau^{\eps_n}$.
Then $U(g)$ is an extended half-tree, and by letting $V(g)$ be the complement extended half-tree resulting from removing the edge connecting $gg_{n+1}^{-1}\tau^{-\eps_n}G$ and $gG$,
the collection of all sets of the form $U(g)$ and $V(g)$, $g\in\Gamma$, generates the shadow topology.

For any $g\in\Gamma$, let $K(g)$ (resp.\ $L(g)$) be the fixator subgroup of the extended half-tree $U(g)$ (resp.\ $V(g)$).
Then $K(g)$ is the fixator of the half-tree $U(g)\cap T$, and $L(g)$ is the fixator of $V(g)\cap T$.

\begin{lemma}\label{stabmerighthere}
Let $\Gamma=\operatorname{HNN}(G,H,\theta)$ be a non-ascending HNN extension with Bass-Serre tree $T$.
For any $g\in\Gamma\setminus G$ with direction $\eps$,
\[
K(g)=gK_{-\eps}g^{-1},\quad L(g)=gt^{-1}\tau^{-\eps}K_{\eps}\tau^{\eps}tg^{-1},
\]
where $t\in G$ is the end letter of $g$.
Moreover, $g\in\Gamma$ fixes all $x\in T$ if and only if $g\in\ker{\Gamma}$.
\end{lemma}
\begin{proof}
Suppose that $g'\in\Gamma$ is a fixator of all elements in $U(g)$ where
\[
g=g_1\tau^{\eps_1}g_2\tau^{\eps_2}\cdots g_n\tau^{\eps_n}g_{n+1}\in\Gamma
\]
in the normal form, and let $\eps$ be the direction of $g$.
Then $g'$ fixes all edges in the subtree spanned by $U(g)$, i.e., $g'(mH)=mH$ for all $m\in\Gamma$ with normal form beginning with $g_1\tau^{\eps_1}g_2\tau^{\eps_2}\cdots g_n\tau^{\eps}$.
Therefore $g'\in K(g)$ if and only if $g'(grH)=grH$ for all $r\in\Gamma$ with the normal form $r=r_1\tau^{f_1}r_2\tau^{f_2}\cdots r_m\tau^{f_m}r_{m+1}$ where either 
\begin{itemize}
\item[(1)] $m=0$, or 
\item[(2)] $m\geq 1$ and either (2a) $f_1=\eps$, or (2b) $f_1=-\eps$ and $r_1\notin H_\eps$.
\end{itemize} In case (2b), the fact $r_1\in S_{-f_1}=S_{\eps}$ implies that $r_1\neq 1$. Therefore the $r\in\Gamma$ of the above form constitute the set $\Gamma\setminus T^\dagger_{-\eps}$, so
\[
g'\in g\left(\bigcap_{r\in\Gamma\setminus T_{-\eps}^\dagger}rHr^{-1}\right)g^{-1}=gK_{-\eps}g^{-1}.
\]
Hence $K(g)=gK_{-\eps}g^{-1}$.

Now let $s=g_1\tau^{\eps_1}g_2\tau^{\eps_2}\cdots g_n\tau^{\eps_{n-1}}g_n$ and notice that $sU(g_n\tau^{\eps_n})\cap T=U(g)\cap T$.
Indeed, $m\in\Gamma$ has normal form beginning with $g_n\tau^{\eps_n}$ if and only $sm$ has normal form beginning with $g_1\tau^{\eps_1}g_2\tau^{\eps_2}\cdots g_n\tau^{\eps_n}$.
In particular, $sV(g_n\tau^{\eps_n})\cap T=V(g)\cap T$, so that $L(g)=sL(g_n\tau^{\eps_n})s^{-1}$.

We next observe that $L(g_n\tau^{\eps_n})$ is the subgroup of elements in $\Gamma$
fixing all vertices of half-trees of the form $U(r\tau^{\eps_n})$ for $r\in S_{-\eps_n}\setminus\{g_n\}$ and $U(t\tau^{-\eps_n})$ for $t\in S_{\eps_n}$, as well as the vertex $G$.
In particular, any $s\in L(g_n\tau^{\eps_n})$ must fix all the edges emanating from and ending in the vertex $G$,
i.e., $\{rH\}_{r\in S_{-1}}$ and $\{t\tau^{-1} H\}_{t\in S_1}$, so that $s\in\bigcap_{g\in G}g(H\cap\theta(H))g^{-1}$. Thus
\begin{align*}
L(g_n\tau^{\eps_n})
&=\bigcap_{r\in S_{-\eps_n}\setminus\{g_n\}}r\tau^{\eps_n}K_{-\eps_n}\tau^{-\eps_n} r^{-1}\cap\bigcap_{t\in S_{\eps_n}} t\tau^{-\eps_n} K_{\eps_n}\tau^{\eps_n}t^{-1}\cap\bigcap_{g\in G}g(H\cap\theta(H))g^{-1}\\
&=\bigcap_{\substack{r\in\Gamma\setminus g_nT_{\eps_n}^\dagger\\\lvert r\rvert\geq 1}}rHr^{-1}\cap\bigcap_{g\in G}g(H\cap\theta(H))g^{-1}\\
&=\bigcap_{\substack{r\in\Gamma\setminus T_{\eps_n}^\dagger\\\lvert r\rvert\geq 1}}g_nrHr^{-1}g_n^{-1}\cap \bigcap_{g\in G}g(H\cap\theta(H))g^{-1}\subseteq g_nK_{\eps_n}g_n^{-1}.
\end{align*}
Conversely, assume that $(g_nr)^{-1}s(g_nr)\in H$ for all $r\in\Gamma\setminus T^\dagger_{\eps_n}$.
Then for any $g'\in\Gamma$ such that $g'G$ belongs to one of the half-trees $U(r\tau^{\eps_n})$ for $r\in S_{-\eps_n}\setminus\{g_n\}$ and $U(t\tau^{-\eps_n})$ for $t\in S_{\eps_n}$,
we have $g_n^{-1}g'\notin\Gamma\setminus T^{\dagger}_{\eps_n}$, so that $sg'G=g'G$, and clearly $s\in G$. Hence $s\in L(g_n\tau^{\eps_n})$.
We conclude that
\[
L(g)=sL(g_n\tau^{\eps_n})s^{-1}=sg_nK_{\eps_n}g_n^{-1}s^{-1}=gg_{n+1}^{-1}\tau^{-\eps_n}K_{\eps_n}\tau^{\eps_n}g_{n+1}g^{-1},
\]
which completes the proof. 
\end{proof}

In the Bass-Serre tree of a non-ascending HNN extension $\Gamma=\operatorname{HNN}(G,H,\theta)$, remove the edge $H$ connecting $G$ and $\tau^{-1}G$.
By the above lemma, the fixator subgroups of the resulting half-trees are $K(\tau^{-1})=\tau^{-1} K_1\tau=\theta(K_1)$ and $L(\tau^{-1})=K_{-1}$. 
Similarly, if we remove the edge $\tau H$ connecting $G$ and $\tau G$,
then the elements that fix the two resulting half-trees are $\theta^{-1}(K_{-1})$ and $K_1$.
By Lemma~\ref{wordlemma}, the normal closures in $\Gamma$ of $K_{-1}$, $\theta(K_1)$, $\theta^{-1}(K_{-1})$, and $K_1$ all coincide.

\begin{proposition}\label{normalsubs}
Let $\Gamma$ be a non-ascending HNN extension with Bass-Serre tree $T$, and consider $T\cup\partial T$ equipped with the shadow topology.
The action of $\Gamma$ on $T$ satisfies
\[
\ker{\Gamma}=\ker (\Gamma\curvearrowright T)=\ker(\Gamma\curvearrowright\partial T)=\ker(\Gamma\curvearrowright\overline{\partial T})
\]
and
\[
\operatorname{int}{\Gamma}=\operatorname{int}(\Gamma\curvearrowright\partial T)=\operatorname{int}(\Gamma\curvearrowright\overline{\partial T}).
\]
\end{proposition}
\begin{proof}
By Lemma~\ref{stabmerighthere}, we have $\ker{\Gamma}=\ker(\Gamma\curvearrowright T)$.
The above paragraph says that both quasi-kernels $K_{-1}$ and $K_1$ are fixator subgroups of half-trees in $T$, so by Corollary~\ref{half-tree-fixators},
their normal closures, or the interior $\operatorname{int}{\Gamma}$, coincides with the interior of the action of $\Gamma$ on $\partial T$.
The remaining identities are explained after the proof of Lemma~\ref{dontstabmehere}.
\end{proof}

\begin{theorem}\label{hnn-utp-cstar}
Let $\Gamma$ be a non-ascending HNN extension with Bass-Serre tree $T$.
Then $\Gamma$ has the unique trace property if and only if $\operatorname{ker}{\Gamma}$ has the unique trace property,
and $\Gamma$ is $C^*$-simple if and only if $\operatorname{int}{\Gamma}$ is $C^*$-simple.
\end{theorem}
\begin{proof}
This follows from Proposition~\ref{normalsubs} and Proposition~\ref{interior-kernel}.
\end{proof}

\begin{theorem}\label{amenablequasis}
Let $\Gamma$ be a non-ascending HNN extension, with quasi-kernels $K_{-1}$ and $K_1$.

If $\ker{\Gamma}$ is trivial, then $\Gamma$ is $C^*$-simple if and only if $K_{-1}$ and $K_1$ are trivial or non-amenable.

If $K_{-1}$ and $K_1$ are amenable, then $\Gamma$ is $C^*$-simple if and only if $K_{-1}$ and $K_1$ are trivial.
\end{theorem}
\begin{proof}
This follows from Proposition~\ref{normalsubs} and Theorem~\ref{fixator-amenable},
since both $K_{-1}$ and $K_1$ are fixator subgroups for half-trees in the Bass-Serre tree $T$.
\end{proof}

\begin{example}
Let $G=\Z$ and let $g\in G$ be a generator of $G$. For $m,n\in\Z\setminus\{0\}$,
define $H=\langle g^m\rangle$ (thus corresponding to $m\Z$) and an injective homomorphism $\theta\colon H\to G$ by $\theta(g^{km})=g^{kn}$ for $k\in\Z$.
Then the HNN extension $\operatorname{HNN}(G,H,\theta)=\operatorname{HNN}(\Z,m\Z,km\mapsto kn)$ is the \emph{Baumslag-Solitar group}
\[
\Gamma=\operatorname{BS}(m,n)=\langle g,\tau\mid\tau^{-1}g^m\tau=g^n\rangle.
\]
Clearly, $\Gamma$ is non-ascending if and only if $\min\{\lvert m\rvert,\lvert n\rvert\}\geq 2$.
A result of the second author \cite[Theorem~4.9]{ivanov2007} states that $\operatorname{BS}(m,n)$ is $C^*$-simple if and only if $\min\{\lvert m\rvert,\lvert n\rvert\}\geq 2$ and $\lvert m\rvert\neq\lvert n\rvert$.

We give a new proof using the $C^*$-simplicity criterion for HNN extensions given above.
Notice first that if $\lvert m\rvert=\lvert n\rvert$, then $H=\langle g^m\rangle$ is a normal abelian subgroup of $\Gamma$.
Furthermore, $\operatorname{BS}(\pm 1,n)$ and $\operatorname{BS}(m,\pm 1)$ are solvable.
Indeed, in the case $m=1$, $N=\langle\{\tau^k g\tau^{-k}\mid k\in\Z\}\rangle$ is a normal abelian subgroup of $\Gamma$, with $\Gamma/N$ infinite and cyclic.

Next assume that $\min\{\lvert m\rvert,\lvert n\rvert\}\geq 2$ and $\lvert m\rvert\neq\lvert n\rvert$ and let $d$ be the greatest common divisor of $m$ and $n$,
so that we may write $m=dm'$ and $n=dn'$ for $m',n'\in\Z$. We may assume that $\lvert n\rvert>\lvert m\rvert$ so that $\lvert n'\rvert>1$.
Clearly $\tau^{-1}H\tau=\langle g^n\rangle=\langle g^{dn'}\rangle$.
For $k\in\Z$ and $i\geq 1$, write $kd(n')^i=qm+r$ for $q\in\Z$ and $0\leq r<m$.
If
\[
G\ni\tau^{-1}g^{kd(n')^i}\tau=g^{qn}\tau^{-1}g^r\tau,
\]
then $r=0$, so $m\mid kd(n')^i$ and $m\mid kd$.
Hence $m(n')^i\mid kd(n')^i=qm$, meaning that $(n')^i$ divides $q$ and 
\begin{gather*}
\tau^{-1}g^{kn}\tau\in\langle g^{n(n')}\rangle=\langle g^{d(n')^2}\rangle,
\tau^{-i}g^{kn}\tau^i\in\langle g^{n(n')^i}\rangle=\langle g^{d(n')^{i+1}}\rangle.
\end{gather*}
This shows that $\tau^{-i}H\tau^{i}\cap G\subseteq\langle g^{d(n')^{i+1}}\rangle$ for $i\geq 1$, meaning that $K_1=\{1\}$.
Therefore $\Gamma$ is $C^*$-simple by Theorem~\ref{hnnpowers} or Theorem~\ref{amenablequasis}.
\end{example}

\section{A non-\texorpdfstring{$C^*$}{C*}-simple HNN extension with the unique trace property}\label{example}

In the following, let $A=\bigcup_{n\in\mathbb{N}}\{0,1\}^{2n}$ and define $X\subseteq A$ by
\[
X=\{(i_1,j_1,\ldots,i_n,j_n)\mid n\in\mathbb{N},\ (j_k,i_{k+1},j_{k+1}) \not\in \{ (0,0,1), (1,0,0) \}\text{ for all }1\leq k\leq n-1\}.
\]
We say that an element $x=(i_1,j_1,\ldots,i_n,j_n)\in X$ has \emph{length} $\ell(x)=n$. For $1\leq k\leq n$, we define $\pi_k\colon X\to\{0,1\}^2$ by
\[
\pi_k(i_1,j_1,\ldots,i_n,j_n)=(i_k,j_k).
\]
We let $\underline{H}$ be the group with generators $\{h(x)\mid x\in X\}$ satisfying the following relations:
\begin{itemize}
\item[(R1)] $h(x)^2=1$ for all $x\in X$;
\item[(R2)] $h(x) h(y) = h(y) h(x)$ whenever $\ell(x)=\ell(y)$;
\item[(R3)] for all $2\leq k+1\leq n$, 
\begin{multline*}
h(i'_1, j'_1, \dots, i'_k, j'_k) h(i_1, j_1, \dots, i_n, j_n) h(i'_1, j'_1, \dots, i'_k, j'_k) = \\ 
  \begin{cases}
        h(i_1, j_1, \dots,i_k, j_k, i_{k+1}\oplus 1, j_{k+1}, \dots, i_n, j_n), & \text{if } i'_1 = i_1, j'_1 = j_1, \dots, i'_k =i_k, j'_k = j_k=j_{k+1}, \\
        h(i_1, j_1, \dots,i_k, j_k, i_{k+1}, j_{k+1}, \dots, i_n, j_n), & \text{otherwise,}
  \end{cases}
\end{multline*}
\end{itemize}
where $\oplus$ denotes addition modulo $2$.
We now let $G$ be the group containing $\underline{H}$ as a subgroup as well as elements $g_{0},g_1$ such that $\langle \underline{H},g_{0},g_1\rangle=G$,
\[
g_{0}^2 = 1 = g_1^2,\quad g_{0} g_1 = g_1 g_{0},
\]
and such that for all $n\in\mathbb{N}_0$,
\begin{itemize}
\item $g_j h(i, j\oplus 1,i_1, j_1, \dots, i_n, j_n) =h(i, j\oplus 1,i_1, j_1, \dots, i_n, j_n)g_j$ and
\item $g_j h(i, j,i_1, j_1, \dots, i_n, j_n) = h(i\oplus 1,j ,i_1, j_1, \dots, i_n, j_n) g_j.$
\end{itemize}
For all $n \in \mathbb{N}$ and $x=(i_1,j_1,\ldots,i_n,j_n)\in X$, let us write 
\begin{multline*}
H(i_1,j_1,\ldots,i_n,j_n)=\{h(x')\mid x'\in X,\ \ell(x')\geq\ell(x),\ \pi_k(x')=\pi_k(x)\text{ for all }k=1,\ldots,n \}.
\end{multline*}
For $i\in\{0,1\}$, we define $H_i = \langle\underline{H},g_i\rangle$ and a map $\theta_i\colon\{h(x)\mid x\in X\}\cup\{g_i\}\to G$ by
\[
\theta_i(g_i) = h(0,i),\quad \theta_i(h(0,i\oplus 1))=g_{i\oplus 1},
\]
and for all $n\in\mathbb{N}_0$, 
\begin{itemize}
\item $\theta_i( h(i_1,j_1,\dots, i_n, j_n) ) =h(0, i, i_1,j_1, \dots, i_n, j_n)$ whenever $(i_1,j_1)\neq(0,i\oplus 1)$;
\item $\theta_i( h(0,i\oplus 1,i_1,j_1, \dots, i_n, j_n) ) = h(i_1, j_1, \dots, i_n, j_n)$ whenever $n\geq 1$.
\end{itemize}
\begin{lemma}
The map $\theta_i$ extends to a group homomorphism $H_i\to G$ with image $H_{i\oplus 1}$.
\end{lemma}
\begin{proof}
Let $x,y\in X$. If $\pi_1(x)\neq\pi_1(y)$, then (R2) and (R3) together imply that $h(x)$ and $h(y)$ commute.
If $\pi_1(x),\pi_1(y)\notin\{(0,i\oplus 1)\}$, then (R2) and (R3) together imply that $\theta_i(h(x))$ and $\theta_i(h(y))$ commute.
We may therefore assume $\pi_1(y)=(0,i\oplus 1)$.
If $\ell(y)\geq 2$, then $\pi_2(y)\neq(0,i)$ so that $\theta_i(h(x))$ and $\theta_i(h(y))$ commute by (R2) and (R3) again, and if $\ell(y)=1$,
then $\theta_i(h(x))\in H(0,i)$ and $\theta(h(y))=g_{i\oplus 1}$ commute.

We may therefore assume from now on that $\pi_1(x)=\pi_1(y)$. If $\ell(x)=\ell(y)$, then by (R2) we have $h(x)h(y)=h(y)h(x)$ and $\theta_i(h(x))\theta_i(h(y))=\theta_i(h(y))\theta_i(h(x))$,
so we may assume that $\ell(x)>\ell(y)$ without loss of generality.
Write $x=(i_1,j_1,\ldots,i_n,j_n)$, so that $\ell(x)=n>k=\ell(y)$. Let $m\leq k$ be largest such that $\pi_r(x)=\pi_r(y)$ for all $1\leq r\leq m$.
If $m<k$, then $h(x)$ and $h(y)$ commute, and so do $\theta_i(h(x))$ and $\theta_i(h(y))$.
Therefore we may assume $m=\ell(y)$. If $j_{k+1}\neq j_k$, then $h(x)$ and $h(y)$ commute, as do $\theta_i(h(x))$ and $\theta_i(h(y))$ once more.
We may therefore also assume $j_{k+1}=j_k$ as well, as well as $\pi_1(y)=\pi_1(x)=(0,i\oplus 1)$ (the case $\pi_1(y)\neq(0,i\oplus 1)$ is completely analoguous).
Then $j_2=i\oplus 1$. There are now two cases:
\begin{itemize}
\item If $k=1$, then
\begin{align*}
\theta_i(h(y))\theta_i(h(x))\theta_i(h(y))&=g_{i\oplus 1}h(i_2,j_2,\dots, i_n, j_n)g_{i\oplus 1}\\&=h(i_2\oplus 1, j_2, \dots, i_n, j_n)\\&=\theta_i(h(x)h(y)h(x)).
\end{align*}
\item If $k\geq 2$, then
\begin{align*}
\theta_i(h(y))\theta_i(h(x))\theta_i(h(y))&=h(i_2, j_2, \dots, i_k, j_k)h(i_2,j_2, \dots, i_n, j_n)h(i_2,j_2, \dots, i_k, j_k)\\&=h(i_2\oplus 1, j_2, \dots, i_n, j_n)=\theta_i(h(x)h(y)h(x)).
\end{align*}
\end{itemize}
Finally, if $j_1=i\oplus 1$, then $g_i$ commutes with $h(x)$, and $\theta_i(g_i)=h(0,i)$ commutes with $\theta_i(h(x))$. If $j_1=i$, then
\begin{align*}
\theta_i(g_i)\theta_i(h(x))\theta_i(g_i)&=h(0,i)h(0,i,i_1,j_1,\ldots,i_n,j_n)h(0,i)\\&=h(0,i,i_1\oplus 1,j_1,\ldots,i_n,j_n)\\&=\theta_i(h(i_1\oplus 1,j_1,\ldots,i_n,j_n))\\&=\theta_i(g_ih(i_1,j_1,\ldots,i_n,j_n)g_i).
\end{align*}
It follows that there is a homomorphism $\theta_i\colon H_i\to G$ with the desired requirements, and $\theta_i(H_i)=\langle \underline{H},g_{i\oplus 1}\rangle$.
\end{proof}

Letting $H=H_0$ and $\theta=\theta_0\colon H\to G$ be the group homomorphism of the above lemma (for $i=0$), then $\theta_1\circ\theta=\textup{id}_H$. In particular, $\theta$ is injective.

The following are some easy properties of the group $G$.
\begin{lemma}\label{onehellofalemma}
\textup{(i)} $[G:H] = [G:\theta(H)] = 2$ and therefore $H \lhd G$ and $\theta(H) \lhd G$; \\ 
\textup{(ii)} $\underline{H} = \langle H(0,0)\rangle \cdot\langle H(1,0)\rangle \cdot\langle H(0,1)\rangle \cdot\langle H(1,1)\rangle$;\\
\textup{(iii)} For each $x\in\{0,1\}^2$ there exists a homomorphism $\underline{H}\to\langle H(x)\rangle$
which is the identity map on $\langle H(x)\rangle$ and maps each element of $H(y)$ to $1$ for all $y\in\{0,1\}^2\setminus\{x\}$.
In particular, $\langle H(x)\rangle\cap\langle H(y)\rangle=\{1\}$ for distinct $x,y\in\{0,1\}^2$.
\end{lemma}
\begin{proof}
(i) is evident and (ii) follows from the commutation relations.
To see that (iii) holds, notice for $x\in\{0,1\}^2$ that the universal property of $\underline{H}$
yields a surjective homomorphism $\varphi_x$ onto the group $H'(x)$ generated by $\{h(y)\mid y\in X\,\ \pi_1(y)=x\}$ with the same relations as $\underline{H}$;
moreover, there is a natural homomorphism $\iota_x\colon H'(x)\to\underline{H}$ such that $\iota_x(h(y))=h(y)$ for all $y\in X$ with $\pi_1(y)=x$.
As $\varphi_x\circ\iota_x$ is the identity map on $H'(x)$, (iii) follows.
\end{proof}

We now consider the HNN extension $\Gamma=\mathrm{HNN}(G,H,\theta)=\langle G,\tau\mid\tau^{-1}h\tau=\theta(h)\rangle$.
We choose the coset representatives $\{1,g_1\}$ for $H$ and $\{1,g_{0}\}$ for $\theta(H)$.
Therefore, any $g\in\Gamma$ is of the form
\[
g=r_1\tau^{\eps_1}\cdots r_n\tau^{\eps_n}r_{n+1},
\]
where $r_i\in\{1,g_1\}$ if $\eps_i=1$ and $r_i\in\{1,g_0\}$ if $\eps_i=-1$, for all $1\leq i\leq n$.

Let $K_{-1}$ and $K_1$ be the quasi-kernels of the HNN extension $\Gamma$.

\begin{lemma}\label{easyquasi}
The quasi-kernels of $\Gamma$ satisfy $\langle H(0,0)\rangle\subseteq K_{-1}$ and $\langle H(0,1)\rangle\subseteq K_1$.
\end{lemma}
\begin{proof}
We first define $M_{-1}=H(0,1)$ and $M_1=H(0,0)$. Let $i\in\{0,1\}$ and $s\in H(0,i)$. Let $\eps=-1$ if $i=0$ and $\eps=1$ if $i=1$.
We claim first that $g^{-1}sg\in M_{\eps'}$ for all $g\in\Gamma\setminus T_{\eps}^\dagger$ of length $\geq 1$, end letter $1$ and direction $\eps'$. 

Assume that $\lvert g\rvert=1$. We separate the cases $i=0$ and $i=1$:
\begin{itemize}
\item $i=0$. For $\eps_1=-1$ we have $r_1=g_0$, in which case $r_1sr_1\in g_0H(0,0)g_0=H(1,0)$ and $g^{-1}sg\in\theta^{-1}(H(1,0))\subseteq H(0,1)$.
If $\eps_1=1$, then $sr_1=r_1s$ and \[g^{-1}sg=\theta(s)\in H(0,0).\]
\item $i=1$. For $\eps_1=-1$ we have $sr_1=r_1s$ and $g^{-1}sg=\theta^{-1}(s)\in H(0,1)$, and for $\eps_1=1$ we have $r_1=g_1$ so that $r_1sr_1\in g_1H(0,1)g_1=H(1,1)$ and \[g^{-1}sg\in\theta(H(1,1))\subseteq H(0,0).\]
\end{itemize}

Next assume that we have proved the claim for all $g\in\Gamma\setminus T_{-1}^\dagger$ with $\lvert g\rvert=n-1$, $n\geq 2$.
Let $g\in\Gamma\setminus T_{-1}^\dagger$ with $\lvert g\rvert=n$ and write
\[
g^{-1}sg=(\tau^{-\eps_n}r_n\cdots\tau^{-\eps_1}r_1)s(r_1\tau^{\eps_1}\cdots r_n\tau^{\eps_n}).
\]
Letting $\eps'$ be the direction of $g'=r_1\tau^{\eps_1}\cdots r_{n-1}\tau^{\eps_{n-1}}$, then $s'=(g')^{-1}sg'\in M_{\eps'}$. 

If $\eps_n=-1$, then for $r_n=1$ we have $\eps_n=\eps'$, so that $s'\in M_{-1}=H(0,1)$.
Hence $g^{-1}sg=\tau s'\tau^{-1}=\theta^{-1}(s')\in H(0,1)$.
If $r_n=g_0$, then for $\eps=1$ we have $r_ns'r_n\in g_0H(0,0)g_0=H(1,0)$ and $g^{-1}sg=\theta^{-1}(H(1,0))\in H(0,1)$,
and for $\eps'=-1$ we have $r_ns'r_n\in g_0H(0,1)g_0\subseteq H(0,1)$ and $g^{-1}sg\in H(0,1)$.

If $\eps_n=1$, the procedure is exactly the same: if $r_n=1$, then $s'\in M_1=H(0,0)$.
Thus $g^{-1}sg=\tau^{-1}s'\tau=\theta(s')\in H(0,0)$.
If $r_n=g_1$, then $\eps'=1$ implies $r_ns'r_n\in g_1H(0,0)g_1=H(0,0)$ and $g^{-1}sg\in\theta(H(0,0))=H(0,0)$, and $\eps'=-1$ implies $r_ns'r_n\in g_1H(0,1)g_1=H(1,1)$ and $g^{-1}sg\in\theta(H(1,1))\subseteq H(0,0)$.
This finishes the proof of the claim.

For any $g\in\Gamma\setminus T_{\eps}^\dagger$ of length $\geq 1$, then there exists $t\in G$ such that $gt\in\Gamma$ has end letter $1$.
Thus $g^{-1}sg\in t(H(0,0)\cup H(0,1))t^{-1}\subseteq tHt^{-1}=H$, since $H$ is normal in $G$. Therefore $H(0,0)\subseteq K_{-1}$ and $H(0,1)\subseteq K_1$.
\end{proof}

To prove the reverse inclusions, we need two preparatory lemmas that describe how we may ``reduce'' a given element in $\{h(x)\mid x\in X\}$ by means of conjugation by elements of $\Gamma$.

\begin{lemma}\label{thereducer}
Let $x\in X$ and $n=\ell(x)$. Define
\[
r_j(x)=\left\{\begin{array}{cl}\tau^{-1}&\text{when }\pi_j(x)=(0,0),\\g_0\tau^{-1}&\text{when }\pi_j(x)=(1,0),\\
\tau&\text{when }\pi_j(x)=(0,1),\\g_1\tau&\text{when }\pi_j(x)=(1,1),\end{array}\right.
\]
and let $r(x)=r_1(x)\cdots r_n(x)$.
Furthermore, let $i\colon X\to\{0,1\}$ be given by
\[
i(x)=\left\{\begin{array}{cl}0&\text{if }\pi_n(x)\in\{(0,0),(1,0)\},\\1&\text{if }\pi_n(x)\in\{(0,1),(1,1)\}.\end{array}\right.
\]
Then $r(x)\in\Gamma$ is the unique element in $\Gamma$ with end letter $1$, such that $\lvert r(x)\rvert=n$ and $r(x)^{-1}h(x)r(x)\in\{g_0,g_1\}$, in which case $r(x)^{-1}h(x)r(x)=g_{i(x)}$.
The resulting map $r\colon X\to\Gamma$ is injective.
\end{lemma}
\begin{proof}
It is easy to check that $r(x)$ and $i(x)$ satisfy the desired requirements.
Notice also that for all $x\in X$, the inner automorphisms $\mathrm{Ad}(\tau),\mathrm{Ad}(g_1\tau),\mathrm{Ad}(\tau^{-1}),\mathrm{Ad}(g_0\tau^{-1})$ map an element of the form $h(y)$, $\ell(y)\geq 2$,
to an element of the form $h(y')$ for which $\ell(y')\in\{\ell(y)\pm 1\}$.

For any $x\in X$, $r(x)=r_1(x)\cdots r_n(x)$ is a reduced word in $\Gamma$.
Indeed, for all $1\leq k\leq n-1$ then $\pi_k(x)=(0,1)$ or $\pi_k=(1,1)$ implies $\pi_{k+1}(x)\neq(0,0)$, meaning $r_k(x)\in\{\tau,g_1\tau\}$ implies $r_{k+1}(x)\in\{g_0\tau^{-1},\tau,g_1\tau\}$.
Similarly, one sees that $r_k(x)\in\{\tau^{-1},g_0\tau^{-1}\}$ implies $r_{k+1}(x)\in\{\tau^{-1},g_0\tau^{-1},g_1\tau\}$.
Hence injectivity of $r$ will follow from Britton's lemma, once we show that $r(x)$ is uniquely determined for all $x\in X$.

For uniqueness, let us first assume that $n=1$.
For $(i,j)\neq (0,0)$ assume that $t=t_1\tau^{\eps_1}$ satisfies $\tau^{-\eps_1}t_1^{-1}h(i,j)t_1\tau^{\eps_1}=g_i$.
As $t_1^{-1}h(i,j)t_1=h(x')$ for some $x'\in X$ of length $1$, then $\tau^{\eps_1}g_i\tau^{-\eps_1}=h(x')$.
Hence $i=0$ implies $\eps_1=-1$, in which case the equation $h(i,j)=t_1h(0,0)t_1^{-1}$ determines $t_1$; if $i=1$, then $\eps_1=1$ and $h(i,j)=t_1h(0,1)t_1^{-1}$ also determines $t_1$.
Hence $t=r_1(x)$.

For $n\geq 2$, write $x=(i_1,j_1,\ldots,i_n,j_n)$ and assume that uniqueness holds for $x'=(i_2,j_2,\ldots,i_n,j_n)$.
If $t=t_1\tau^{\eps_1}\cdots t_n\tau^{\eps_n}$ satisfies $t^{-1}h(x)t=g_i$ for some $i$, then \[(t_1\tau^{\eps_1})^{-1}h(x)(t_1\tau^{\eps_1})=h(x').\]
Now $g_i=t^{-1}h(x)t=(t_2\tau^{\eps_2}\cdots t_n\tau^{\eps_n})^{-1}h(x')(t_2\tau^{\eps_2}\cdots t_n\tau^{\eps_n})$,
so that uniqueness yields $r(x')=t_2\tau^{\eps_2}\cdots t_n\tau^{\eps_n}$ and $t=t_1\tau^{\eps_1}r(x')$.
Now $(t_1\tau^{\eps_1})^{-1}h(x)(t_1\tau^{\eps_1})=h(i_2,j_2,\ldots,i_n,j_n)$ determines $t_1$ and $\eps_1$, completing the proof.
\end{proof}
\begin{lemma}\label{thereducer2}
Let $x,y\in X$ such that either of the following conditions hold:
\begin{itemize}
\item[(i)] $\ell(x)<\ell(y)$;
\item[(ii)] $\ell(y)\leq\ell(x)$ and $\pi_i(x)=\pi_i(y)$ for all $1\leq i\leq k$ implies $k<\ell(y)$.
\end{itemize}
Then for the element $r(x)\in\Gamma$ defined in Lemma~\ref{thereducer}, we have $r(y)^{-1}h(x)r(y)=h(x')$ for some $x'\in X$. In fact, in either of the above cases, the sequence $x'$ satisfies
\[
\ell(x')=\ell(x)+\ell(y)-2\max\{k\mid \pi_i(x)=\pi_i(y)\text{ for all }i\leq k\}.
\]
\end{lemma}
\begin{proof}
Assume first that we have proved that the equation holds for $x,y\in X$ such that $\pi_1(y)\neq\pi_1(x)$.
If $\pi_1(y)=\pi_1(x)$, let $1\leq k\leq \min\{\ell(x),\ell(y)\}$ be largest such that $\pi_i(y)=\pi_i(x)\text{ for all }i\leq k$.
Remember that we deliberately choose $r_i(x)$ for $1\leq i\leq\ell(x)$ such that applying the inner automorphisms $\mathrm{Ad}(r_1(x)^{-1}),\mathrm{Ad}(r_2(x)^{-1}),\ldots,\mathrm{Ad}(r_{\ell(x)}(x)^{-1})$
in that order removes the first tuple, then the second tuple, until we have removed all tuples from $x$.
Hence the inner automorphism $\mathrm{Ad}(r_1(x)\cdots r_k(x))^{-1}$ will remove the first $k$ tuples from both $h(x)$ and $h(y)$.

If $\ell(y)\leq\ell(x)$, then $k<\ell(y)$. Let $x_0$ and $y_0$ be $x$ and $y$ respectively,
but without their first $k$ tuples, and observe that $r(x_0)^{-1}h(y_0)r(x_0)=h(y')$ for some $y'\in X$ such that $\ell(y')=\ell(y_0)+\ell(x_0)=\ell(y)+\ell(x)-2k$.
We then have
\begin{align*}
r(x_0)^{-1}h(y_0)r(x_0)&=r(x_0)^{-1}\varphi(h(y))r(x_0)\\&=r(x_0)^{-1}(r_1(x)\cdot r_k(x))^{-1}h(y)r_1(x)\cdot r_k(x)r(x_0)\\&=r(x)^{-1}h(y)r(x).
\end{align*}
If $\ell(y)>\ell(x)$ and $k<\ell(x)$, then the same argument applies;
finally, for $k=\ell(x)$, then the first $k$ tuples of $y$ constitute $x$, so that the $y'\in X$ such that $r(x)^{-1}h(y)r(x)^{-1}$ has length $\ell(y)-\ell(x)=\ell(x)+\ell(y)-2k$.

Now assume that $\pi_1(y)\neq\pi_1(x)$. Write $x=(i_1,j_1,\ldots,i_n,j_n)$ and define $r_0(x)=1$.
Define $\varphi_i=\mathrm{Ad}(r_i(x)^{-1})$ for all $0\leq i\leq n$ and $\Phi_i=\varphi_i\circ\cdots\circ\varphi_0$, so that $\Phi_n(h(x))=r(x)^{-1}h(x)r(x)$.
We claim for all $0\leq i\leq n$ that the first tuple of $\Phi_i(h(y))$ never agrees with the first tuple of $\Phi_i(h(x))$.
This will prove that $h(y)$ gets one tuple longer for each $i$, while $h(x)$ will eventually collapse to either $g_0$ or $g_1$. 

Let $1\leq i\leq n$. The first line in one of the 12 cells in the table below is one of the possibilities of what the first tuple of $\Phi_{i-1}(h(x))$ could be, the arrow pointing to what the first tuple of $\Phi_i(h(x))$ may be.
The next line considers $\Phi_{i-1}(h(y))$ with distinct first tuple from $\Phi_{i-1}(h(x))$, the arrow pointing to what the first tuple of $\Phi_i(h(y))$ will then be (which depends on the first tuple of $\Phi_{i-1}(h(x))$).
Notice that the pair of two digits after the arrow on the lower line inside each cell never coincides with any pair above it.

\begin{table}[!h]\begin{tabular}{|ccc|ccc|ccc|}\hline
10&$\rightarrow$&00,10,11	&01&$\rightarrow$&01,10,11	&11&$\rightarrow$&01,10,11\\
00&$\rightarrow$&01				&00&$\rightarrow$&00				&00&$\rightarrow$&00\\\hline
00&$\rightarrow$&00,10,11	&10&$\rightarrow$&00,10,11	&11&$\rightarrow$&01,10,11\\
01&$\rightarrow$&01				&01&$\rightarrow$&01				&01&$\rightarrow$&00\\\hline
00&$\rightarrow$&00,10,11	&01&$\rightarrow$&01,10,11	&11&$\rightarrow$&01,10,11\\
10&$\rightarrow$&01				&10&$\rightarrow$&00				&10&$\rightarrow$&00\\\hline
00&$\rightarrow$&00,10,11	&01&$\rightarrow$&01,10,11	&10&$\rightarrow$&00,10,11\\
11&$\rightarrow$&01				&11&$\rightarrow$&00				&11&$\rightarrow$&01\\\hline
\end{tabular}\end{table}
This completes the proof.
\end{proof}

\begin{lemma}\label{thequasis}
The quasi-kernels of $\Gamma$ are given by $K_{-1}= \langle H(0,0)\rangle$ and $K_1= \langle H(0,1)\rangle$.\end{lemma}
\begin{proof}
We established the inclusions ``$\supseteq$'' in Lemma~\ref{easyquasi}.
Showing that $K_{-1}\subseteq \langle H(0,0)\rangle$ is equivalent to showing that for all $g\notin \langle H(0,0)\rangle$, there exists $r\in \Gamma\setminus T_{-1}^\dagger$ such that $r^{-1}gr\notin H$.

Of course, if $g\notin H$, just take $r=e$. If $g=hg_0$ for $h\in\underline{H}$, then take $r=g_0\tau^{-1}$, so that $\tau g_0gg_0\tau^{-1}=\tau g_0\tau^{-1}\theta^{-1}(h)\notin H$.
Since $H=\underline{H}\cup g_0\underline{H}$, we may assume from this point on that $g\in\underline{H}$.
Due to Lemma~\ref{onehellofalemma}, we may write \[g=\prod_{x\in\{0,1\}^2}g_x,\] where $g_x$ is a product of the form $\prod_{k=1}^{n_x}h(\omega_{x,k})$ where $\pi_1(\omega_{x,k})=x$.
We can assume that $n_x$ is the smallest positive integer such that $g_x$ is a product of $n_x$ elements from $H(x)$.

Let $x\in\{0,1\}^2$. By means of the commutation relations in $\underline{H}$, we may assume that $\ell(\omega_{x,k})\leq\ell(\omega_{x,k+1})$ for all $1\leq k<n_x$.
If $n_x\geq 2$, assume that $\omega_{x,k}=\omega_{x,m}$ for some $1\leq k<m\leq n_x$. Then $\ell(\omega_{x,k})=\ell(\omega_{x,k'})$ for all $k\leq k'\leq m$,
so that
\[
h(\omega_{x,k})h(\omega_{x,k+1})\cdots h(\omega_{x,m})=h(\omega_{x,k})^2h(\omega_{x,k+1})\cdots h(\omega_{x,m-1}),
\]
contradicting the assumption that $n_x$ is smallest possible.
Hence $\omega_{x,k}\neq \omega_{x,m}$, so that $h(\omega_{x,k})\neq h(\omega_{x,m})$ for all distinct $k,m\in\{1,\ldots,n_x\}$.

Since $h(\omega_{x,k})$ and $h(\omega_{y,k'})$ commute for distinct $x,y\in\{0,1\}^2$ and $k,k'$, we can reorder the factors and write \[g=h(x_1)h(x_2)\cdots h(x_n),\]
where $n$ is smallest possible, $\ell(x_i)\leq\ell(x_{i+1})$ for all $1\leq i\leq n-1$ and $h(x_i)\neq h(x_j)$ for all $i\neq j$ by what we saw above.
Since $g\notin\langle H(0,0)\rangle$, there is a smallest $1\leq i\leq n$ such that $\pi_1(x_i)\neq (0,0)$.
Then the element $r(x_i)\in\Gamma$ of Lemma~\ref{thereducer} satisfies $r(x_i)^{-1}h(x_i)r(x_i)=g_m$ for some $m\in\{0,1\}$, and $r(x_i)\in\Gamma\setminus T_{-1}^\dagger$.
Moreover, $r(x_i)^{-1}h(x_j)r(x_i)\in\underline{H}$ for all $j\neq i$ by Lemma~\ref{thereducer2},
since the existence of $1\leq j<i$ such that $\pi_k(x_j)=\pi_k(x_i)$ for all $1\leq k\leq\ell(\omega_j)$ would contradict the minimality of $i$.
Hence \[r(x_i)^{-1}gr(x_i)\in \underline{H}g_m\underline{H}=g_m\underline{H}.\]
If $m=0$, let $r=r(x_i)g_0\tau^{-1}$, and if $m=1$, let $r=r(x_i)g_1\tau$.

The proof that $K_1\subseteq\langle H(0,1)\rangle$ is completely analoguous.
\end{proof}

\begin{lemma}\label{locfin}
The group $G$ is locally finite and hence amenable.
\end{lemma}
\begin{proof}
For any non-negative integer $n\geq 0$, consider the subgroup $\underline{H}_n$ of $\underline{H}$ generated by the union of the sets $H_{x,n}=\{H(y)\mid y\in X,\ \pi_1(y)=x,\ \ell(y)\leq n\}$, $x\in\{0,1\}^2$.
As in the proof of Lemma~\ref{thequasis}, any element in the latter subgroup can be decomposed into factors $g_x$ for $x\in\{0,1\}^2$.
By assuming the number of factors in each $g_x$ to be the smallest possible, we can write $g_x=\prod_{s\in H_{x,n}}s^{\eps_s}$ for numbers $\eps_s\in\{0,1\}$, $s\in H_{x,n}$.
Therefore there are only finitely many elements in $\langle H_{x,n}\rangle$, and since the sets $H_{x,n}$ and $H_{y,n}$ commute in $\underline{H}$ for $x\neq y$, it follows that $\underline{H}_n$ is finite.

Every finitely generated subgroup of $G$ is contained in the subgroup $G_n$ generated by $g_0$, $g_1$ and $\underline{H}_n$ for some $n\geq 0$.
As $\underline{H}_n$ is invariant under conjugation by $g_0$ and $g_1$, any element in $G_n$ has a unique decomposition $g=g_0^{\eps_0}g_1^{\eps_1}h$ for $\eps_0,\eps_1\in\{0,1\}$ and $h\in\underline{H}_n$,
and so $G_n$ is also finite.
\end{proof}

\begin{theorem}
The HNN extension $\Gamma$ has the unique trace property, but is not $C^*$-simple.
\end{theorem}
\begin{proof}
Since $K_{-1}=\langle H(0,0)\rangle$ and $K_1=\langle H(0,1)\rangle$ by Lemma~\ref{thequasis}, we see that
\[
\ker{\Gamma}=\langle H(0,0)\rangle\cap\langle H(0,1)\rangle=\{1\}
\]
by Lemma~\ref{onehellofalemma}~(iii), meaning that $\Gamma$ has the unique trace property by Theorem~\ref{hnn-utp-cstar}.
However, $K_{-1}$ and $K_1$ are both amenable by Lemma~\ref{locfin} and non-trivial, so $\Gamma$ is not $C^*$-simple by Theorem~\ref{amenablequasis}.
\end{proof}

\begin{remark}
The interior of $\Gamma$ coincides with the normal closure of $G$ in $\Gamma$.
Indeed, $H(1,0)=\tau H(0,0,1,0)\tau^{-1}$, $H(1,1)=\tau H(0,0,1,1)\tau^{-1}$, $g_0=\tau h(0,0)\tau^{-1}$, and $g_1=\tau^{-1}h(0,1)\tau$ all belong to the normal closure of $\langle H(0,0),H(0,1)\rangle$.
Thus $G\subseteq\operatorname{int}{\Gamma}$, so the normal closure of $G$ belongs to $\operatorname{int}{\Gamma}$.
Conversely, $H(0,0),H(0,1)\subseteq G$, and the containment clearly still holds when passing to normal closures.
Moreover, consider the map $\Gamma\to\Z$ defined by $g\mapsto 0$ for $g\in G$ and $\tau\mapsto 1$.
The normal closure of $G$ coincides with the kernel of this map, and it follows that $\Gamma$ is isomorphic to a semidirect product $(\operatorname{int}{\Gamma})\rtimes\Z$.
\end{remark}

\begin{remark}
The tree of $\Gamma$ is the regular tree of branching degree~$4$.
Let $T_0$ denote the subtree consisting of all vertices $gG$, where $g\in\Gamma\setminus G$ starts with either $g_0\tau^{-1}$ or $\tau^{-1}$ on reduced form (that is, of type $-1$),
and let $T_1$ denote the subtree consisting of all vertices $gG$, where $g\in\Gamma\setminus G$ starts with either $g_1\tau$ or $\tau$ on reduced form (that is, of type $1$).
Then we can find a coloring of the edges of $T$ in such a way that the local permutation of $g_0$ is $(12)$ on $T_0$ and $\textup{id}$ on $T_1$,
the local permutation of $g_1$ is $\textup{id}$ on $T_0$ and $(34)$ on $T_1$,
and the local permutation of $\tau$ is $(23)$ everywhere (cf.~\cite{leboudec}).
Therefore $\Gamma$ is not one of the examples of \cite[Theorem~C]{leboudec}.
\end{remark} 

\begin{proposition}
The group $\Gamma$ is amenablish.
\end{proposition}

\begin{proof}
Let $\Gamma^+$ be the subgroup of $\Gamma$ generated by the stabilizers of the edges of $T$,
and let $T_0$ and $T_1$ be as in the previous remark.
It is easy to see that the fixator subgroup of $T_0$ is $\theta(H)$ and the fixator subgroup of $T_1$ is $H$.
For example the edge $H$ with ends $G$ and $\tau G$ is stabilized by the subgroup $[H \cap \theta(H)] \times \tau [H \cap \theta(H)] \tau^{-1}$ which is the product of the fixators of the corresponding half-trees.
This statement holds for every edge.
It is clear that if $g\in\Gamma$, the stabilizer subgroup of the edge $grH$,
where $r \in \{ 1, g_1, \tau^{-1}, g_0 \tau^{-1} \}$ is $\gamma [H \cap \theta(H)] \gamma^{-1} \times g r [H \cap \theta(H)] r^{-1} g^{-1}$ and these groups generate $\operatorname{int}\Gamma$.
We can use \cite[Corollary~4.6]{leboudec-prescribed}, which is modeled after the Tits criterion,
to conclude that a normal subgroup $N < \operatorname{int}\Gamma$ contains the commutator groups $[\Gamma_e, \Gamma_e]$ for every $e \in E(T)$.
It easily follows that
\[
\underline{\operatorname{int}\Gamma} := \langle g [ H \cap \theta(H), H \cap \theta(H) ] g^{-1} \mid g \in \Gamma \rangle
\]
consists of all elements of $\operatorname{int}\Gamma$ with even number of $g_0$ and even number of $g_1$ in their products.
Therefore
\[
\underline{\operatorname{int}\Gamma} = \operatorname{ker}(\Gamma \to \mathbb{Z} \times \mathbb{Z}_2 \times \mathbb{Z}_2),
\]
where the homomorphism is defined on the generators of $\Gamma$ as follows
\[
\tau \mapsto (1,0,0), \quad g_0 \mapsto (0,1,0), \quad g_1 \mapsto (0,0,1).
\]
We conclude that $\underline{\operatorname{int}\Gamma}$ is a simple amenablish group.
Since the class of amenablish groups is closed under extensions, it follows that $\Gamma$ is amenablish.
\end{proof}

\begin{appendix}

\section{Compactness of \texorpdfstring{$T\cup\partial T$}{an extended tree}}\label{compactness}

We include a proof of the following result of Monod and Shalom from \cite[Section~4.1]{monodshalom}.
\begin{theorem}\label{compact4ever}
Let $T$ be a tree with boundary $\partial T$. Then $T\cup\partial T$ is compact and totally disconnected in the shadow topology.
Moreover, $T\cup\partial T$ is metrizable if and only if $T$ is countable.
\end{theorem}
In \cite[Section~4.1]{monodshalom} countability of $T$ is assumed, but it is not needed to conclude compactness of the space $T\cup\partial T$.
Throughout this section, let $T$ denote a tree and for any edge $e$ in $T$, we let $Z_0(e)\subseteq T$, $Z_B(e)\subseteq\partial T$ and $Z(e)$ be as defined in Section~\ref{treesect}.
\begin{lemma}\label{elementary-tree-lemma}
Let $e,e'$ be edges in $T$.
\begin{itemize}
\item[(i)] If $Z_B(e)\cap Z_B(e')\neq\oo$, then $Z_0(e)\cap Z_0(e')$ is infinite. In particular, $Z_0(e)\cap Z_0(e')=\oo$ if and only if $Z(e)\cap Z(e')=\oo$.
\item[(ii)] $T\cup\partial T$ is the disjoint union of the clopen sets $Z(e)$ and $Z(\ov{e})$.
\end{itemize}
\end{lemma}
\begin{proof}
(i): Any $x\in Z_B(e)\cap Z_B(e')$ is the equivalence class of a ray $(r_i)_{i\geq 0}$ such that $r_0=s(e)$ and $r_1=r(e)$, which is cofinal to a ray $(s_i)_{i\geq 0}$ such that $s_0=s(e')$ and $s_1=r(e')$.
Now there exist $k,n\geq 0$ such that $Z_0(e)\ni r_{i+k}=s_{i+n}\in Z_0(e')$ for all $i\geq 0$.

(ii): Clearly $Z_0(e)$ and $Z_0(\ov{e})$ are disjoint (so that $Z(e)\cap Z(\ov{e})=\oo$ by (i)), and their union is $T$.
Therefore let $x\in\partial T$ and assume that $x\notin Z_B(e)$. Let $(r_i)_{i\geq 0}$ be the ray in the equivalence class of $x$ such that $r_0=s(e)$.
Then $r_1\neq r(e)$, so that we may define a ray $(t_i)_{i\geq 0}$ by $t_i=r_{i-1}$ for $i\geq 1$ and $t_0=r(e)$. Now $(t_i)$ and $(r_i)$ are cofinal, and $(t_i)\in Z_\infty(\ov{e})$, meaning that $x\in Z_B(\ov{e})$.
\end{proof}

\begin{lemma}\label{here-comes-the-funk}
Suppose that two edges $e$ and $e'$ in $T$ satisfy $d(r(e),r(e'))>d(s(e),s(e'))$. Then $Z(e)\cap Z(e')=\oo$.
\end{lemma}
\begin{proof}
We clearly have $e\neq e'$ and $e\neq \ov{e'}$. If we had $\{s(e),r(e)\}\subseteq Z_0(e')$,
then the geodesic between $r(e')$ and $s(e)$ would not contain $r(e)$ (otherwise $d(s(e),s(e'))=1+d(s(e),r(e'))=d(r(e),r(e'))+2$),
meaning that $d(r(e),r(e'))=d(s(e),r(e'))+1=d(s(e),s(e'))$.
Hence $\{s(e),r(e)\}\subseteq Z_0(\ov{e'})$.

Now $r(e)$ is not contained in the geodesic between $s(e)$ and $r(e')$ -- otherwise \[d(s(e),s(e'))=d(r(e),s(e'))+1=d(r(e),r(e')),\] a contradiction.
Therefore the geodesic between $r(e)$ and $r(e')$ contains $s(e)$.
As $r(e')\notin Z_0(e)$ (otherwise $d(s(e'),s(e))<d(r(e'),r(e))<d(r(e'),s(e))$, so that $s(e)\in Z_0(e')$),
then for all $v\in Z_0(e)$,
\[
d(v,r(e'))=d(v,s(e))+d(s(e),r(e'))=d(v,s(e))+d(s(e),s(e'))+1>d(v,s(e')),
\]
i.e., $v\notin Z_0(e')$. Hence $Z_0(e)$ and $Z_0(e')$ are disjoint, so the claim follows from Lemma~\ref{elementary-tree-lemma}~(i).
\end{proof}

The following lemma is proved in \cite[Lemma~6.4.9]{serre}, so we omit the proof.

\begin{lemma}[``The bridge lemma'']\label{bridgelemma}
If $T_1$ and $T_2$ are subtrees of a tree $T$, and $T_1\cap T_2$ contains at most one vertex, there are unique vertices $v_1\in T_1$, $v_2\in T_2$ with $d(v_1,v_2)=d(T_1,T_2)$.
Moreover, we have $d(w_1,w_2)=d(w_1,v_1)+d(v_1,v_2)+d(v_2,w_2)$ for all $w_1\in T_1$ and $w_2\in T_2$.
\end{lemma}

\begin{proof}[Proof of Theorem~\ref{compact4ever}]
Let $X=T\cup\partial T$, and let $\mS\subseteq 2^X$ be the Boolean algebra generated by the extended shadows $\{Z(e)\}_{e\in E}$.
We consider the space $\mM$ of Boolean homomorphisms $\mS\to\{0,1\}$ of the compact Hausdorff space $2^\mS$,
i.e., maps $\mu\colon\mS\to\{0,1\}$ such that $\mu(\oo)=0$, $\mu(X)=1$, and
\[
\mu(A\cup B)=\max\{\mu(A),\mu(B)\},\  \mu(A\cap B)=\min\{\mu(A),\mu(B)\},\quad  A,B\in\mS.
\]
We claim that $\mM$ is compact.
For any $\mu\in 2^\mS$,
define maps $\varphi_1(\mu)\colon \mS\times\mS\to\{0,1\}$ by $\varphi_1(\mu)(A,B)=\mu(A\cup B)$ and $\varphi_2(\mu)\colon\mS\times\mS\to\{0,1\}$ by $\varphi_2(\mu)(A,B)=\max\{\mu(A),\mu(B)\}$.
These maps are continuous, when viewed as maps $\varphi_1 \colon 2^\mS \to  2^{\mS\times\mS}$ and $\varphi_2 \colon 2^\mS \to  2^{\mS\times\mS}$.
Since $2^{\mS\times\mS}$ is Hausdorff, $\{\mu\in 2^\mS\mid \varphi_1(\mu)=\varphi_2(\mu)\}$ is closed.
Considering intersections the same way, it follows that $\mM$ is closed in $2^\mS$.

Now let $\mu\in\mM$, and let $\mathcal O_\mu=\{e\in E\mid \mu(Z(e))=1\}$.
Then $\mathcal O_\mu$ is an orientation of $T$ by Lemma~\ref{elementary-tree-lemma}~(ii).
We claim that there exists $x\in X$ such that \[\{x\}=\bigcap_{e\in\mathcal O_\mu}Z(e).\]

Observe that for any $v\in T$ there is at most one edge $e$ such that $s(e)=v$ and $e\in\mathcal O_\mu$.
(If $s(e)=s(e')=v$ for $e\neq e'$, then $Z_0(e)\cap Z_0(e')=\oo$, but then either $\mu(Z(e))=0$ or $\mu(Z(e'))=0$.)
There are now two cases:

\medskip

\emph{Case 1: Any $v\in T$ admits one edge $e_v\in\mathcal O_\mu$ such that $s(e_v)=v$.}
For fixed $r_0\in T$, let $e_0\in\mathcal O_\mu$ such that $s(e_0)=r_0$, let $r_1=r(e_0)$.
We then let $e_1\in\mathcal O_\mu$ such that $s(e_1)=r_1$, and continue taking edges $(e_i)_{i\geq 0}$ and vertices $(r_i)_{i\geq 0}$ in $T$, until we obtain a sequence $(r_i)_{i\geq 0}$ of vertices in $T$.
Then $(r_i)_{i\geq 0}$ is a ray in $T$: if $r_{i-1}=s(e_{i-1})=r(e_i)=r_{i+1}$ and $r(e_{i-1})=s(e_i)=r_i$, then $\ov{e_i}=e_{i-1}\in\mathcal O_\mu$, a contradiction.
Let $x\in\partial T$ be the equivalence class of $(r_i)_{i\geq 0}$.

Now $x\in Z(e)$ for all $e\in\mathcal O_\mu$. If $e=e_i$ for some $i\geq 0$, the claim is evident, assume that $e\notin\{e_i\mid i\geq 0\}$.
Now let $k\geq 0$ be such that $r_k$ is the closest of the vertices in the ray $(r_i)_{i\geq 0}$ to the subtree of endpoints $s(e),r(e)$ of $e$ by the bridge lemma (Lemma~\ref{bridgelemma}).
If $s(e)$ belongs to the geodesic between $r(e)$ and $r_k$, then $d(r(e),r(e_k))=1+d(r(e),s(e_k))=2+d(s(e),s(e_k))$, meaning that $Z(e)\cap Z(e_k)=\oo$, contradicting Lemma~\ref{here-comes-the-funk}.
Hence we may let $(t_i)_{i\geq 0}$ be a ray emanating from $s(e)$, passing through $r(e)$ and the vertices $(r_i)_{i\geq k}$, meaning that $x\in Z(e)$.

We next note that $\bigcap_{n=0}^\infty Z(e_n)=\{x\}$.
Indeed, for $v\in T$, let $n\geq 0$ such that $d(r_n,v)\leq d(r_j,v)$ for all $j\geq 0$ by the bridge lemma.
Then $d(v,s(e_n))=d(v,r_n)<d(v,r_{n+1})=d(v,r(e_n))$ and $v\notin Z(e_n)$.
If $y\in\partial T$ satisfies $y\in Z(e_n)$ for all $n\geq 1$, let $(t_i)_{i\geq 0}$ be the unique ray with equivalence class $y$ and $t_0=r_0$.
For $n\geq 1$, $(t_i)_{i\geq 0}$ is cofinal to a ray $(u_j)$ such that $u_{j+1}=r_{n+1}$ and $u_j=r_n$ for some $j\geq 0$.
In particular, there exists $m\geq 0$ for which $t_m\in Z_0(e_n)$.
Since $t_0\notin Z_0(e_n)$, the geodesic from $t_0=r_0$ to $t_m$ passes through $r_n$ and $r_{n+1}$, meaning that $t_n=r_n$ and $t_{n+1}=r_{n+1}$.
Since $n$ was arbitrary, $y=x$. Therefore $\{x\}=\bigcap_{e\in\mathcal O_\mu}Z(e)$.

\medskip

\emph{Case 2. There exists $v\in T$ such that $s(e)\neq v$ for all $e\in\mathcal O_\mu$.}
For some $e\in\mathcal O_\mu$, assume that $v\in Z(\ov{e})$, so that the geodesic from $v$ to $r(e)$ passes through $s(e)$.
Let $e_0,\ldots,e_n=e$ be the edges constituting this geodesic, so that $r(e_i)=s(e_{i+1})$ for all $0\leq i\leq n-1$.
Since $s(e_{n-1})\notin Z(e_n)$, the geodesic from $w\in Z(e_n)$ to $s(e_{n-1})$ must contain $r(e_{n-1})=s(e_n)$, meaning that $Z(e_n)\subseteq Z(e_{n-1})$.
Hence $e_{n-1}\in\mathcal O_\mu$. We continue this way until we obtain $e_0\in\mathcal O_\mu$, a contradiction (since $s(e_0)=v$).
Hence $v\in Z(e)$ for all $e\in\mathcal O_\mu$.

Assume that $w\in T\setminus\{v\}$. Let $e_0,\ldots,e_n$ be a path such that $r(e_i)=s(e_{i+1})$ for all $1\leq i\leq n-1$ and $s(e_0)=v$, $r(e_n)=w$.
Then $e_0\notin\mathcal O_\mu$, so that $\ov{e_0}\in\mathcal O_\mu$ but $w\notin Z(\ov{e_0})$.
For $x\in\partial T$, let $(r_i)_{i\geq 0}$ be the unique ray representing $x$ such that $r_0=v$, and let $(e_i)_{i\geq 0}$ be the edges such that $s(e_i)=r_i$, $r(e_i)=r_{i+1}$.
Now $x\in Z(e_0)$, so that $x\notin Z(\ov{e_0})$ by Lemma~\ref{elementary-tree-lemma}~(ii), but $\ov{e_0}\in\mathcal O_\mu$ as above.
Hence $\{v\}=\bigcap_{e\in\mathcal O_\mu}Z(e)$.

\medskip

Finally, let $\mS_\mu=\{S\in\mS\mid \mu(S)=\delta_x(S)\}$.
Then $S_\mu$ is stable under complements, finite unions and intersections and $Z(e)\in S_\mu$ for all $e\in E$, so $\mS_\mu=\mS$ by minimality. Hence $\mu=\delta_x$.

We conclude that $\mM=\{\delta_x \mid  x\in X\}$.
For edges $e_1,\ldots,e_n$ in $T$ such that $Z=\bigcap_{i=1}^nZ(e_i)$ is non-empty (this is a basis element of $X$), then $U=\{f\in 2^\mS \mid  f(Z)=1\}$ is open in $2^\mS$, so that $\mM\cap U$ is open in $\mM$.
As $\delta_x\in\mM\cap U$ if and only if $x\in Z$, it follows that $\mM\cap U$ is the pre-image of the clopen basis element $Z$ under the bijection $\mM\to X$ given by $\delta_x\mapsto x$,
so that this map is continuous. In particular $X$ is compact, and since $X$ is Hausdorff, the bijection $\mM\to X$ is a homeomorphism.

Notice that $\mM$ is totally disconnected by $2^\mS$ being totally disconnected.
If $T$ is countable, then $\mS$ is countable, so that $2^\mS$ and $\mM$ are metrizable.
If $\mM$ is metrizable, it has a countable basis  $\mc{B}$ of clopen subsets.
By \cite[Corollary, p.~75]{halmos} the countable algebra generated by $\mc{B}$ contains all clopen sets of $\mM$, i.e., the edge set of $T$ is countable.
Hence $T$ is countable.
\end{proof}

\end{appendix}

\bibliographystyle{amsplain}

\end{document}